\documentclass[oneside, 10pt]{amsart}
\usepackage{amscd}
\usepackage{amsfonts}
\usepackage{amsmath}
\usepackage{amsthm, amstext, amssymb, comment, enumitem, verbatim, mathtools, xfrac, microtype, nameref, thmtools, tikz, tikz-cd}
\usepackage[breaklinks=true]{hyperref}
\usepackage[capitalize]{cleveref}
\usepackage[matrix,arrow,curve]{xy}
\usepackage{longtable}
\usepackage{mathrsfs}
\usetikzlibrary{shapes, arrows, calc, arrows.meta, fit, positioning, decorations.markings}
\usepackage[natbib=true, backend=biber, firstinits=true, style=numeric, sortlocale=en_US, url=false, doi=false, eprint=true, maxbibnames=4]{biblatex}
\usepackage{mathtools}
\usepackage{tikz}
\renewbibmacro*{volume+number+eid}{\ifentrytype{article}{\- \iffieldundef{volume}{}{Vol.~\printfield{volume},}\iffieldundef{number}{}{ No.~\printfield{number},}}}
\renewbibmacro{in:}{\ifentrytype{article}{}{\printtext{\bibstring{in}\intitlepunct}}}
\newbibmacro{string+doi}[1]{\iffieldundef{doi}{\iffieldundef{url}{#1}{\href{\thefield{url}}{#1}}}{\href{https://dx.doi.org/\thefield{doi}}{#1}}}
\DeclareFieldFormat[article, inproceedings, inbook, book, online]{title}{\usebibmacro{string+doi}{\mkbibquote{#1}}}

\DeclareMathOperator{\St}{St}
\DeclareMathOperator{\GL}{GL}
\DeclareMathOperator{\E}{E}
\DeclareMathOperator{\G}{G}
\DeclareMathOperator{\UU}{U}
\newcommand{\Gsc}{\G_\mathrm{sc}}
\newcommand{\Esc}{\E_\mathrm{sc}}

\DeclareMathOperator{\Um}{Um}

\DeclareMathOperator{\K}{K}

\DeclareMathOperator{\Ker}{Ker}
\DeclareMathOperator{\Stab}{Stab}
\DeclareMathOperator{\Img}{Im}

\newcommand{\rA}{\mathsf{A}}

\newcommand{\rD}{\mathsf{D}}
\newcommand{\rE}{\mathsf{E}}

\newcommand{\ZZ}{\mathbb{Z}}
\newcommand{\StW}{\widehat{W}}
\newcommand{\StH}{\widehat{H}}
\newcommand{\inv}{^{-1}}
\newcommand{\RNum}[1]{\uppercase\expandafter{\romannumeral #1\relax}}
\newcommand{\eval}[4]{\mathrm{ev} \textstyle \left[\frac{#2[#1] \rightarrow #3}{#1 \mapsto #4}\right]}
\newcommand{\ev}[3]{\eval{X}{#1}{#2}{#3}}

\newtheorem{thm}{Theorem}
\Crefname{thm}{Theorem}{Theorems}
\numberwithin{equation}{section}
\numberwithin{thm}{section}

\newtheorem{lemma}[thm]{Lemma}
\numberwithin{lemma}{section}
\Crefname{lemma}{Lemma}{Lemmas}
\newlist{lemlist}{enumerate}{1} \setlist[lemlist]{label=\textrm{(\arabic{lemlisti})}, ref=\thelemma.(\arabic{lemlisti}),noitemsep} \Crefname{lemlisti}{Lemma}{Lemma}

\newtheorem{cor}[thm]{Corollary}
\Crefname{cor}{Corollary}{Corollaries}

\newtheorem{prop}[lemma]{Proposition}
\Crefname{prop}{Proposition}{Propositions}

\newtheorem*{thm*}{Theorem}
\newtheorem*{lemma*}{Lemma}

\theoremstyle{definition}

\newtheorem{dfn}[lemma]{Definition}
\Crefname{dfn}{Definition}{Definitions}
\newtheorem{example}[lemma]{Example}

\Crefname{example}{Example}{Examples}

\theoremstyle{remark}

\newtheorem{rem}[lemma]{Remark}
\Crefname{rem}{Remark}{Remarks}

\title{On the $\mathbb{A}^1$-invariance of $\K_2$ of Chevalley groups of simply-laced type}
\keywords {Steinberg group, $K_2$-functor, Serre problem {\em Mathematical Subject Classification (2010):} 19C20}
\author {Sergei Sinchuk}
\dedicatory{In the memory of my teacher Nikolai Vavilov}
\email {sinchukss at gmail.com}
\date {\today}
\begin{document}
\maketitle

\begin{abstract}
In this paper we study the $\mathbb{A}^1$-invariance of the unstable functor $\mathrm{K}_2(\Phi, R)$
in the case when $\Phi$ is an irreducible root system of type $\mathsf{ADE}$ containing $\rA_4$ and not of type $\rE_8$.
We show that in the geometric case, i.\,e. when $R$ is a regular ring containing a field $k$
one has $\K_2(\Phi, R[t]) = \K_2(\Phi, R)$, which allows one to interpret the unstable $\K_2$ groups
as $\mathbb{A}^1$-fundamental groups of Chevalley--Demazure group schemes in the $\mathbb{A}^1$-homotopy category.
We also prove a variant of "early stability" theorem which allows one to find a generating set
of $\K_2(\Phi, A[X_1, \ldots X_n])$ in the case when $A$ is a Dedekind domain.
\end{abstract}

\section{Introduction}\label{sec:introduction}

Let $R$ be a commutative ring with a unit and $\Phi$ be a reduced irreducible root system of rank $\geq 2$.
Recall that a choice of a lattice $\Lambda$ lying between the root lattice $\mathbb{Z}\Phi$ and the weight lattice $P(\Phi)$
specifies a Chevalley--Demazure group scheme $G=\G_\Lambda(\Phi, -)$ over $\ZZ$.
The Steinberg group $\St(\Phi, R)$ is a certain group defined by means of generators and relations, see~\cref{subsec:steinberg-preliminaries} for the definition.
By definition, the \textit{unstable $\K_2$-functor} $\K_2^G(R)$ is defined as the kernel of the canonical homomorphism $\St(\Phi, R) \to \G_\Lambda(\Phi, R)$.

This work continues the series of papers~\cite{LS20, LSV2} and has a twofold aim.
The first goal of this work is to generalize the main result of~\cite{LSV2} concerning the $\mathbb{A}^1$-invariance of the functor $\K_2^G(R)$.
Specifically, we prove the following invariance theorem, which generalizes~\cite[Theorem~1.1]{LSV2}:

\begin{thm}[The $\K_2$-analogue of Lindel--Popescu theorem] \label{thm:LP-for-K2}
Let $k$ be an arbitrary field and $R$ be a regular ring containing $k$.
Let $\Phi$ be an irreducible simply-laced root system containing $\rA_4$ but not of type $\rE_8$.
Then for any lattice $\Lambda$ as above and $G = \mathrm{G}_\Lambda(\Phi, -)$ one has
$\K_2^G(R[t])\cong\K_2^G(R).$
\end{thm}

The primary corollary of \cref{thm:LP-for-K2} is the following result, which interprets the unstable $\K_2$-functors as fundamental groups of group schemes
$\mathrm{G}_\Lambda(\Phi, -)$ in the motivic category of F.~Morel and V.~Voevodsky extending the result of~\cite[Corollary~1.2]{LSV2}.

\begin{cor} \label{cor:motivic-pi1} Let $\Phi$, $\Lambda$ and $k$ be as in~\cref{thm:LP-for-K2}.
Then for any $k$-smooth algebra $R$ and $G = \mathrm{G}_{\Lambda}(\Phi, -)$ one has
\[ \pi_1^{\mathbb{A}^1}(G)(R) := \mathrm{Hom}_{\mathscr{H}_{\bullet}(k)}(S^1 \wedge \mathrm{Spec}(R)_+, G) = \mathrm{KV}_2^{G}(R) = \K_2^G(R).\]
\end{cor}
In the above statement $\mathscr{H}_\bullet(k)$ denotes the pointed unstable $\mathbb{A}^1$-homotopy category over $k$, see~\cite{MV99}.
$\mathrm{KV}_2^{G}(R)$ denotes the second Karoubi--Villamayor K-functor, i.\,e. the fundamental group of
the simplicial group $G(R[\Delta^\bullet])$, see e.\,g.~\cite[\S~3]{Jar83} or~\cite[\S~3.2]{LSV2}.

The second primary objective of this paper is to establish a result that describes the structure of unstable $\K_2$-groups
over multivariate polynomial rings over Dedekind domains.
Conceptually, this should be compared to the solution of the Bass--Quillen conjecture for one-dimensional rings, as discussed in~\cite[\S~V.3]{Lam10},
and to analogous results for the unstable $\K_1$-functor, e.\,g.~\cite[Theorem~1.1]{St-Ded}.
Thus, we prove the following

\begin{thm} \label{cor:dedekind}
Let $\Phi$ be as in the statement of~\cref{thm:LP-for-K2} and let $A$ be a Dedekind domain.
Then for any $n \geq 0$ the group $\K_2(\rA_4, A)$ surjects onto $\K_2(\Phi, A[X_1,\ldots, X_n])$.
In particular, $\K_2(\Phi, A[X_1,\ldots, X_n]) = \K_2(\Phi, A)$.
Moreover, if $\K_2(A)$ is generated by Dennis--Stein symbols then for any $n \geq 0$ so is $\K_2(\Phi, A[X_1,\ldots, X_n])$.
\end{thm}

By analogy, \cref{thm:LP-for-K2} can be seen as a counterpart for the functor $\K_2^G$ to the theorem of H.~Lindel and D.~Popescu,
 which resolved the classical Bass--Quillen conjecture in the geometric case—that is, for commutative regular $k$-algebras of arbitrary dimension.
We highlight here that the main novelty of \cref{thm:LP-for-K2} is its applicability to root system types $\Phi = \rE_6, \rE_7$,
as well as its extension to orthogonal groups over the base ring, without requiring that $2$ be invertible in $R$.

In the special case when $R$ is a field, homotopy invariance type results for $\K_2$ similar to~\cref{thm:LP-for-K2} have been known since 1970's, cf. e.\,g.~\cite{Hur77, Re75, VW16}.
Notice also that the homotopy invariance for the unstable $\K_1$-functors and their interpretation as groups of connected components of $G$ in the motivic homotopy category
have been recently obtained in much greater generality by A.~Stavrova, see~\cite[Theorem~1.3]{St-poly}, \cite[Theorem~5.2]{St22}.
Notice also that the Nisnevich sheaf $\pi_1^{\mathbb{A}^1}(G)$ associated to the presheaf of fundamental groups $\pi_1^{\mathbb{A}^1}(G)(R)$
has been recently computed by F.~Morel and A.~Sawant for all split reductive $G$ without any assumptions on the rank of $G$, see~\cite[Theorem~1]{MS23}.
For a more detailed historical account of the problem, we refer the reader to the introduction of~\cite{LSV2}.

The main technical ingredient in the proof of all our results is the refined Horrocks theorem for $\K_2$,~\cref{thm:horrocks-k2},
which generalizes~\cite[Theorem~1]{LS20} and~\cite[Proposition~4.3]{Tu83}.
As noted in the introduction of~\cite{LSV2}, the restrictive assumptions on $\Phi$ and $R$ in these two theorems were a major obstacle to proving a more
general result in~\cite{LSV2}, a gap this paper seeks to fill.
Our proof of~\cref{thm:horrocks-k2} follows the same general approach as that of~\cite[Proposition~4.3]{Tu83},
with the key difference being our use of more modern tools, such as the amalgamation theorem for affine Steinberg groups introduced by A.~Allcock in~\cite{A13}.
For a more detailed discussion of the ideas behind our proof of~\cref{thm:horrocks-k2} and its key differences from~\cite{LS20, Tu83}, we refer the reader to~\cref{sec:horrocks}.

\subsection{Acknowledgements}
This work was supported by the Russian Science Foundation grant №19-71-30002.

\section{Preliminaries}\label{sec:preliminaries}

In this paper all rings are assumed to be commutative and all commutators are left-normed i.\,e. $[a, b] = a b a^{-1} b^{-1}$.

Throughout this paper we denote by $\lambda_a$ the homomorphism of principal localization at $a \in A$.
Also for a prime ideal $M \trianglelefteq A$ we denote by $\lambda_M \colon A \to A_M$ the homomorphism inverting the multiplicative subset $A \setminus M$.

For an $A$-algebra $B$ and $b\in B$ we denote by $\mathrm{ev}_{X=b} \colon A[X]\rightarrow B$ the homomorphism of $A$-algebras evaluating each polynomial
$p(X)\in A[X]$ at $b$, i.e. $\mathrm{ev}_{X=b}(p(X)) = p(b)$.
Sometimes in order to emphasize the choice of $B$, we also use the more explicit notation $\ev{A}{B}{b}$ for the same homomorphism.

\subsection{Steinberg groups, $\K_2$-groups and symbols}\label{subsec:steinberg-preliminaries}
Let $\Phi$ be a finite root system of rank $\ell > 1$.
We assume that $\Phi$ is embedded into $\mathbb{R}^\ell$ whose scalar product we denote by $(\text{-}, \text{-})$.
We also fix some system of simple roots $\Pi = \{\alpha_1, \ldots, \alpha_\ell\} \subset \Phi$.
For a root $\alpha\in\Phi$ we denote by $m_i(\alpha)$ the $i$-th coefficient in the expansion of $\alpha$ in $\Pi$,
i.\,e. $\alpha = \sum_{i=1}^n m_i(\alpha) \alpha_i$.
We denote by $\Phi^+$ (resp. $\Phi^-$) the system of positive (respectively, negative) roots with respect to the basis $\Pi$.

We denote by $\Phi^\vee$ the corresponding dual root system, which, by definition, consists of all coroots $\alpha^\vee = \frac{2}{(\alpha, \alpha)} \alpha$, where $\alpha \in \Phi$.
We denote by $P(\Phi)$ the \textit{weight lattice}, i.\,e. the integral lattice spanned by the \emph{fundamental weights $\varpi_i$}.
Recall that the fundamental weights $\varpi_i$ are uniquely determined by the property $(\varpi_i, \alpha_j^\vee) = \delta_{ij}$.
For $v \in \mathbb{R}^\ell$ define the product $\langle v, \beta \rangle \coloneqq (v, \beta^\vee).$
Clearly, for $\alpha,\beta \in \Phi$ the product $\langle \alpha, \beta\rangle$ is always an integer.
We also denote by $W(\Phi)$ the Weyl group of $\Phi$, i.\,e. the isometry group of $\mathbb{R}^\ell$ generated by fundamental reflections $s_\alpha$,
where $s_\alpha(v) = v - \langle v, \alpha\rangle \alpha$.

By definition, a weight $\omega \in P(\Phi)$ is called a \textit{microweight} if $\langle \omega, \alpha \rangle \in \{ 0, 1 \}$ for all $\alpha \in \Phi^+$.

Now let $R$ be an arbitrary commutative ring with $1$ and suppose that $\Phi$ is a root system whose irreducible components have rank $\geq 2$.
Recall that to the pair $(\Phi, R)$ can associate an abstract group $\St(\Phi, R)$, called the \textit{Steinberg group} of type $\Phi$ over $R$.
By definition, $\St(\Phi, R)$ is the group presented by generators $x_\alpha(a)$, $a \in R$, $\alpha \in \Phi$ and an explicit list of relations (see e.\,g in.~\cite{Ma69, St71}).
In this paper we restrict ourselves to the case when the root system $\Phi$ in question is \textit{simply-laced} (i.\,e. has type $\mathsf{ADE}$),
in which case the defining relations of $\St(\Phi, R)$ reduce to the following shorter list, cf. e.\,g.~\cite[\S~2.2]{S15}:
\begin{align}
x_{\alpha}(a)\cdot x_{\alpha}(b) & =x_{\alpha}(a+b), \tag{R1} \label{x-additivity}\\
[x_{\alpha}(a),\,x_{\beta}(b)]   & =x_{\alpha+\beta}(N_{\alpha,\beta} \cdot ab),\text{ for }\alpha+\beta\in\Phi, \tag{R2} \label{R2} \\
[x_{\alpha}(a),\,x_{\beta}(b)]   & =1,\text{ for }\alpha+\beta\not\in\Phi\cup0. \tag{R3} \label{R3}
\end{align}
The coefficients $N_{\alpha,\beta}$ in the above formula are integers equal to $\pm 1$, they coincide with the structure constants of the complex Lie algebra of type $\Phi$.

The construction of $\St(\Phi, R)$ is obviously functorial in $R$.
For a ring homomorphism $ f\colon A \to B$ we denote by $f^*$ the corresponding homomorphism of Steinberg groups $\St(\Phi, A) \to \St(\Phi, B)$.

Throughout this paper we denote by $\Gsc(\Phi, R)$ the group of points of the simply-connected Chevalley--Demazure group of type $\Phi$ over $R$,
 i.\,e. the group scheme corresponding to the choice $\Lambda = P(\Phi)$.
We denote by $\Esc(\Phi, R)$ the \textit{elementary subgroup} of $\Gsc(\Phi, R)$, i.\,e. the subgroup generated by elementary root unipotents of $\Gsc(\Phi, R)$.
Notice that in~\cite{Vav09, VP} the notation $x_\alpha(a)$ is used to denote the elementary root unipotents.
To prevent confusion we will use different notation $t_\alpha(a)$ for them and reserve the notation $x_\alpha(a)$ solely for generators of Steinberg groups.

Recall that the map sending $x_\alpha(a)$ to $t_\alpha(a)$ gives rise to a well-defined homomorphism $\pi \colon \St(\Phi, R) \to \G_\mathrm{sc}(\Phi, R)$, cf.~\cite[\S~1A]{St78}.
The cokernel and the kernel of $\pi$ are called \textit{the unstable $\K_1$- and $\K_2$-functors modeled on the root system $\Phi$}:
\begin{equation} \label{eq:K1-K2-sequence}
\xymatrix{ 1 \ar[r] & \K_2(\Phi, R) \ar[r] & \St(\Phi, R) \ar[r]^{\pi} & \Gsc(\Phi, R) \ar[r] & \K_1(\Phi, R) \ar[r] & 1}
\end{equation}

In the special case $\Phi =\rA_\ell$ the Steinberg $\St(\rA_\ell, R)$ coincides with the linear Steinberg group $\St(\ell + 1, R)$, while
$\Gsc(\rA_\ell, R) = \mathrm{SL}(\ell + 1, R)$, $\Esc(\rA_\ell, R) = \E(\ell + 1, R)$.

For a \textit{special} subset of $\Phi$ (i.\,e. a subset $S \subseteq \Phi$ such that $S \cap -S = \varnothing$)
we denote by $\UU(S, R)$ the subgroup of $\St(\Phi, R)$ generated by $x_\alpha(a)$, where $a \in R$, $\alpha \in S$.
Notice that the restriction of $\pi$ to $\UU(S, R)$ is injective, so $\UU(S, R)$ is isomorphic to the image $\pi(\UU(S, R)) \leq \Gsc(\Phi, R)$, cf.~\cite[\S~1A]{St78}.
To simplify notation for unipotent groups, we slightly abuse language by using $\UU(S, R)$ to denote subgroups of both $\Gsc(\Phi, R)$ and $\St(\Phi, R)$.

Denote by $\Sigma^+_i$ (resp. $\Sigma^-_i$) the subset of $\Phi$ consisting of $\alpha \in \Phi$ such that $m_i(\alpha) > 0$ (resp. $m_i(\alpha) < 0$).
Denote by $\Delta_i$ the set of roots $\alpha \in \Phi$ such that $m_i(\alpha) = 0$.
It is clear that $\Phi = \Sigma_i^- \sqcup \Delta_i \sqcup \Sigma_i^+$ and that $\Sigma_i^+$, $\Sigma_i^-$ are special subsets of $\Phi$.

For $1 \leq k \leq \ell$ denote by $P_k^\pm$ the subgroup of $\St(\Phi, A)$ generated by $x_\alpha(a)$ for all $\alpha \in \Delta \sqcup \Sigma^\pm_k$ and $a \in A$.
Notice that by~\eqref{R2}--\eqref{R3} the subgroup $P_k^\pm$ admits an analogue of \textit{Levi decomposition}:
\begin{equation} \label{eq:levi-decomp} P_k^\pm = L_k \ltimes U_k^\pm, \end{equation}
where $L_k = \mathrm{Im}(\St(\Delta_k, A) \to \St(\Phi, A))$ and $U_k^\pm = \UU(\Sigma_k^\pm, A)$
(as before, we use the notation $U_k^\pm$ to denote subgroups of both $\St(\Phi, A)$ and $\Gsc(\Phi, A)$).

Following~\cite{Ma69} for $\alpha\in\Phi$ and $u \in R^\times$ we define the following elements of $\St(\Phi, R)$:
\begin{align} w_\alpha(u) & =  x_\alpha(u) \cdot x_{-\alpha}(-u^{-1}) \cdot x_\alpha(u), \label{eq:w-definition} \\
h_\alpha(u) & =  w_\alpha(u) \cdot w_\alpha(-1).  \label{eq:h-definition} \end{align}
The subgroup generated by $w_\alpha(u)$ (resp. $h_\alpha(u)$) for all $\alpha\in \Phi$, $u \in R^\times$ is denoted by $\StW(\Phi, R)$ (resp. $\StH(\Phi, R)$).
By~\cite[Lemme~5.2]{Ma69} $\StH(\Phi, R)$ is a normal subgroup of $\StW(\Phi, R)$.

Denote by $W(\Phi, R)$ the image of $\StW(\Phi, R)$ in the Chevalley group $\Gsc(\Phi, R)$.
It is clear that $W(\Phi, R)$ is contained in the group-theoretic normalizer $N(\Phi, R)$ of the split maximal torus $H(\Phi, R) \leq \Gsc(\Phi, R)$.

In the sequel we will need the following explicit elements of the group $\K_2(\Phi, R)$.
Recall that for arbitrary $u, v \in R^\times$ one defines the \textit{Steinberg symbol} via the formula
\begin{equation} \label{eq:steinberg} \{ u, v \}_\alpha = h_\alpha(uv) \cdot h_\alpha^{-1}(u) \cdot h_\alpha^{-1}(v). \end{equation}
Recall also from~\cite[Lemme~5.4]{Ma69} that
\begin{equation} \label{eq:steinberg-2} [h_\alpha(u), h_\beta(v)] = \{u, v^{\langle \alpha, \beta \rangle}\}_\alpha. \end{equation}
Steinberg symbols depend only on the length of the root $\alpha$, cf.~\cite[pp.~26--28]{Ma69}.
In particular, in the case when $\Phi$ is of simply-laced type, they do not depend on the choice of $\alpha$, which allows us to omit it from notation.

Steinberg symbols are central elements of $\St(\Phi, R)$.
Our assumptions on $\Phi$ guarantee that Steinberg symbols are antisymmetric and bimultiplicative, i.\,e. they satisfy the following identities, cf.~\cite[Lemme~2.4]{Ma69}:
\begin{equation} \label{eq:symbol-properties} \{ u, st \} = \{ u, s\} \{ u, t \}, \ \{ u, v \} = \{ v, u\}^{-1}. \end{equation}

Steinberg symbols generate $\K_2(\Phi, A)$ if $A$ happens to be a field or a local ring, see~\cite[Theorem~2.13]{Ste73}.

Let $A$ be a local ring with maximal ideal $M$ and residue field $\kappa = A/M$.
For $w \in \StW(\Phi, A)$ we denote by $\overline{w}$ the image of $w$ in $W(\Phi)$ under the following composite homomorphism:
\[ \StW(\Phi, A) \to W(\Phi, A) \to W(\Phi, \kappa) \to N(\Phi, \kappa) \to N(\Phi, \kappa) / H(\Phi, \kappa) \cong W(\Phi). \]

\subsection{Relative Steinberg groups} \label{subsec:another-presentation}
In this paper we use the concept of a \textit{relative Steinberg group} introduced by F.~Keune and J.-L.~Loday in~\cite{Ke78, Lo78}.
We will only briefly mention the definition and basic properties of these groups and refer the reader to~\cite[\S~2.3]{LS20} for a more detailed exposition.

Let $R$ be a commutative ring, $I \trianglelefteq R$ be an ideal and let $p$ denote the canonical projection $R \to R/I$.
Denote by $D_{R, I}$ the pullback of two copies of $p$ i.\,e. the ring $R \times_{R/I} R$.
Elements of $D_{R, I}$ are pairs $(a; b)$ such that $a-b \in I$.
We also denote by $p_1$, $p_2$ the canonical projections $D_{R, I} \to R$ and by $p_1^*$, $p_2^*$ the corresponding homomorphisms of Steinberg groups induced by them.
Recall from~\cite[Definition~2.5]{LS20} that \textit{the relative Steinberg group} $\St(\Phi, R, I)$ is defined as the quotient
$\Ker(p_1^*) / C$, where $C = [\Ker(p_1^*), \Ker(p_2^*)]$.
Denote by $\mu$ the homomorphism $\St(\Phi, R, I) \to \St(\Phi, R)$ induced by $p_2^*$.
Also denote by $C(\Phi, R, I)$ the kernel of $\mu$.
It is clear that $C(\Phi, R, I) = \Ker(p_1^*) \cap \Ker(p_2^*) / C$.
Thus, we obtain an exact sequence
\begin{equation}
\xymatrix{C(\Phi, R, I) \ar@{^{(}->}[r] & \St(\Phi, R, I) \ar[r]^\mu & \St(\Phi, R) \ar@{->>}[r]^-{p^*} & \St(\Phi, R/I). }\label{eq:relative-Steinberg}
\end{equation}
Alternatively, the group $\St(\Phi, R, I)$ can be defined via generators and relations as an $\St(\Phi, R)$-group, cf.~\cite[Proposition~6]{S15}
or even as an abstract group, see~\cite{V22}.

\begin{dfn}\label{dfn:crossed-module}
Let $N$ be a group acting on itself by right conjugation.
Let $M$ be a group with a right action of $N$.
Recall that a group homomorphism $\mu\colon M \to N$ is called a \textit{precrossed module} if $\mu$ preserves the action of $N$, i.\,e.
\[\mu(m^n) = \mu(m)^n, \text{for all $m \in M$, $n\in N;$} \]
If, in addition, $\mu$ satisfies the so-called \textit{Peiffer identity}, i.\,e.
\begin{equation}\label{eq:peiffer} {m}^{\mu(m')} = {m'}^{-1} m m', \text{for all $m, m' \in M$,}\end{equation}
then $\mu$ is called a \textit{crossed module}.
\end{dfn}

It is clear from definitions that $\St(\Phi, R)$ acts on $\St(\Phi, R, I)$ by right conjugation.
We claim that, in fact, the following stronger assertion holds.
\begin{lemma} \label{lem:rel-Steinberg-crossed-module}
For any irreducible root system $\Phi$ of rank $\geq 3$ the homomorphism $\mu \colon \St(\Phi, R, I) \to \St(\Phi, R)$ is a crossed module.
\end{lemma}
\begin{proof}
Denote by $\Delta$ the diagonal homomorphism $R \to D_{R, I}.$
By~\cite[Proposition~6]{Lo78} it suffices to show that $[\Ker(p_1^*) \cap \Ker(p_2^*), \Delta^*(\St(\Phi, R))] = 1$.
But the latter is clear since $\Ker(p_1^*) \cap \Ker(p_2^*)$ is contained in the subgroup $\K_2(\Phi, D_{R, I})$,
which is a central subgroup of $\St(\Phi, D_{R, I})$ by~\cite{LSV1}.
\end{proof}

Recall that an ideal $I \trianglelefteq R$ is called \textit{a splitting ideal} if the canonical projection $R \twoheadrightarrow R/I$ splits.
If $I \trianglelefteq R$ is a splitting ideal then $C(\Phi, R, I)$ is trivial and $\St(\Phi, R)$ decomposes into the semidirect product
$\St(\Phi, R, I) \rtimes \St(\Phi, R/I)$, see~\cite[\S~1]{LS17}.

We define the relative group $\K_2(\Phi, R, I)$ as the kernel of the homomorphism $\pi \mu \colon \St(\Phi,R, I) \to \Gsc(\Phi, R)$.
For a special subset of roots $S \subseteq \Phi$ we denote by $\UU(S, I)$ the subgroup of $\St(\Phi, R, I)$ generated by $x_\alpha((0;m))C$, $m \in I$, $\alpha \in S$.
It is clear that $\UU(S, I) \cap \K_2(\Phi, R, I) = 1$.
Notice also that $C(\Phi, R, I) \subseteq \K_2(\Phi, R, I)$.

We will need a relative analogue of Steinberg symbol.
Let $A$ be a local unital ring with maximal ideal $M$ embedded as a subring into a larger unital ring $R$ (e.g. $R = A[X]$ or $R = A[X, X\inv]$).
Since $A$ is local the subset $1+M \subseteq A$ forms a group under multiplication.
We denote this group by $(1+M)^\times$.
Clearly, it is isomorphic to the abelian group $(M, \circ)$ with the operation given by $m \circ m' = m + m' + mm'$.
Now for $a \in R^\times$ and $m \in M$ we denote by $\{a, 1+m\}_r$ the coset $\{(a; a), (1; 1+m)\}C \in \St(\Phi, R, MR)$.
It is clear that the map $1+m \mapsto \{a, 1+m\}_r$ specifies a group homomorphism
\begin{equation} \label{eq:relative-symbol} \{ a, -\}_r \colon (1+M)^\times \to \K_2(\Phi, R, MR). \end{equation}
It is also clear that $\mu(\{a, 1+m\}_r) = \{a, 1+m\}$.

\begin{lemma}\label{lem:symbols}
Assume that $A$ is a local domain with maximal ideal $M$.
Denote by $R$ the Laurent polynomial ring $A[X, X\inv]$.
Then the intersection of the image of the relative symbol map $\{X, -\}_r$ with $C(\Phi, R, M[X, X\inv])$ is trivial.
\end{lemma}
\begin{proof}
Set $F = \mathrm{Frac}(A)$.
Consider the following diagram:
\[\begin{tikzcd}
(1+M)^\times \ar[hookrightarrow, rr] \arrow{d}[swap]{\{X {,}\, -\}_r} \ar[rd, dashed, "\{X{,}\, -\}"] &  & F^\times \ar[hookrightarrow, d, "\{X{,}\, -\}"] \\
\K_2(\Phi ,R, M[X^{\pm 1}]) \ar[r] & \K_2(\Phi, R) \ar[r] & \K_2(\Phi, F[X^{\pm 1}]).
\end{tikzcd}\]
Since the right vertical arrow is injective by~\cite[Lemma~2.2]{LS20}, so is the dashed arrow,
 which implies the assertion of the lemma.
\end{proof}

In the sequel we will need the following Steinberg-level analogue of Bruhat decomposition.
Denote by $\overline{\St}(\Phi, R, I)$ the kernel of $p^*$ (or what is the same, the image of $\mu$).
\begin{lemma}\label{lem:bruhat}
Let $A$ be a local ring with maximal ideal $M$ and residue field $\kappa$.
Let $\Phi$ be any irreducible root system and $\Pi$ and $P$ be a pair of systems of simple roots in $\Phi$.
Denote by $\Phi^+_\Pi, \Phi^+_{P} \leq \Phi$ the systems of positive roots of $\Phi$ corresponding to $\Pi$ and $P$, respectively.
Then the Steinberg group $\St(\Phi, A)$ admits the following decomposition:
\begin{equation}\label{eq:bruhat}
\St(\Phi,A) =\UU(\Phi^+_{P}, A)\cdot \StW(\Phi, A) \cdot \UU(\Phi^+_{\Pi}, A) \cdot \overline{\St}(\Phi, A, M).
\end{equation}
\begin{comment}
Moreover, if for some $u,u'\in \UU(\Phi^{+}_{P}, A)$, $v, v' \in \UU(\Phi^+_{\Pi}, A)$, $w,w'\in \StW(\Phi, A)$ and $l,l'\in \overline{\St}(\Phi, A, M)$ one has
\begin{equation} \label{eq:bwb-eq} uwvl=u'w'v'l', \end{equation} then
$\overline{w}=\overline{w'} \in W(\Phi)$ and $w^{-1}w' \in \StH(\Phi, A) \cap \overline{\St}(\Phi, A, M)$.
\end{comment}
\end{lemma}
\begin{proof}
Notice that $\E(\Phi, \kappa) = \Gsc(\Phi, \kappa)$ is a group with a BN-pair by~\cite[\S~4]{Ge17}.
Since $\K_2(\Phi, \kappa)$ is generated by Steinberg symbols, it is a subgroup of $\StW(\Phi, \kappa)$.
This allows us to lift the BNB-decomposition from $\Gsc(\Phi, \kappa)$ first to $\St(\Phi, \kappa)$ and then to $\St(\Phi, A)$,
which proves the assertion of the lemma in the special case $\Pi = P$.
To obtain the general case choose $w \in \StW(\Phi, A)$ such that ${}^w \UU(\Phi^+_\Pi, A) = \UU(\Phi^+_{P}, A)$ and then multiply both sides of~\eqref{eq:bruhat} by $w$ on the left.
\begin{comment}
Let us verify the second assertion.
Projecting the equality~\eqref{eq:bwb-eq} to $\Gsc(\Phi, \kappa)$ and using the fact that double B-cosets are disjoint we obtain that
$\overline{w} = \overline{w'}$.
Since $\K_2(\Phi, \kappa)$ is generated by Steinberg symbols, the image of $w^{-1}w'$ in $\St(\Phi, \kappa)$ belongs to $\StH(\Phi, \kappa)$,
hence
\begin{multline*} w^{-1}w' \in \left(\StH(\Phi, A) \cdot \overline{\St}(\Phi, A, M)\right) \cdot \StW(\Phi, A) \subseteq \\ \subseteq
\left( \UU(\Phi^+, M) \cdot \StH(\Phi, A) \cdot \UU(\Phi^-, M)\right) \cap \StW(\Phi, A) \subseteq \StH(\Phi, A). \end{multline*}
Here the first inclusion follows from~\cite[Theorem~2.4]{Ste73} and the second inclusion is clear from consideration of matrices.

\[w^{-1} w' \in \bigl(\StH(\Phi, A) \cdot \overline{\St}(\Phi, A, M)\bigr)\cap \StW(\Phi, A) \subseteq \StH(\Phi, A).\]
Here we use the fact that $\K_2(\Phi, A)$ is generated by Steinberg symbols, see e.\,g.~\cite[Theorem~2.5]{Ste73}.
\end{comment}
\end{proof}

In the sequel we will also need the following generating set for the group $\St(\Phi, R, I)$.
\begin{lemma}\label{lem:relative-generators}
For irreducible $\Phi$ of rank $\ell \geq 2$ the group $\St(\Phi, R, I)$ is generated as an abstract group by elements $z_\alpha(m, a) = x_\alpha(m)^{x_{-\alpha}(a)}$, $m \in I$, $a \in R$.

In fact, a stronger assertion holds.
Denote by $\Sigma$ the unipotent radical $\Sigma^+_i$ or $\Sigma^-_i$ corresponding to any $1 \leq i \leq \ell$.
Then the group $\St(\Phi, R, I)$ is generated as an abstract group by the following smaller set of generators consisting of the following two families of elements:
\begin{itemize}
\item $x_{\alpha}(m) = x_\alpha((0; m))C$, where $\alpha \in \Phi$, $m \in I$;
\item $z_{\beta}(m, a) = x_\beta((0; m))^{x_{-\beta}((a; a))}C$, where $\beta \in \Sigma$, $a \in A$, $m \in I$.
\end{itemize}
\end{lemma}
\begin{proof}
It suffices to prove only the second assertion.
Notice that by~\cite[Lemma~4]{S15} the group $\Ker(p_1^*)$ is generated by $x_{\beta}((0; m))^{x_{-\beta}((a; a))}$ for $\alpha \in \Sigma$ and $x_{\alpha}((0; m))$ for $\alpha \in \Phi$.
Consequently, the same is true for $\St(\Phi, R, I)$ i.\,e. $\St(\Phi, R, I) = \Ker(p_1^*)/C$.
\end{proof}

\subsection{Weight automorphisms}\label{subsec:weight-automorphisms}
Recall that for every weight $\omega \in P(\Phi)$ and $\beta \in \ZZ \Phi$ the product $\langle \omega, \beta\rangle$ is an integer.
Thus, a choice of $u \in R^\times$ and $\omega \in P(\Phi)$ specifies an automorphism of the generating set for $\St(\Phi, R)$ (i.\,e. an autobijection) via the following mapping:
\begin{equation}\label{eq:chi-def} x_\alpha(a) \mapsto x_\alpha(u^{\langle\omega, \alpha\rangle} \cdot a),\ \alpha\in \Phi,\ a \in R. \end{equation}
It is not hard to check that this automorphism is compatible with relations~\eqref{x-additivity}--\eqref{R3} and therefore specifies a well-defined automorphism of $\St(\Phi, R)$, which we denote by $\chi_{\omega, u}$.

In the following lemma we check that an analogue of $\chi_{\omega, u}$ can also be defined for relative Steinberg groups.
\begin{lemma} \label{lem:relative-chi}
Let $R$ be a commutative ring, $I$ be its ideal and let $u \in R^\times$.
For every weight $\omega \in P(\Phi)$ there exists a well-defined automorphism $\widetilde{\chi}_{\omega, u}$ of the relative Steinberg group $\St(\Phi, R, I)$ making the following diagram commute:
\[\begin{tikzcd} \St(\Phi, R, I) \ar[r, "\widetilde{\chi}_{\omega, u}"] \ar[d] & \St(\Phi, R, I) \ar[d] \\
\St(\Phi, R) \ar[r, "\chi_{\omega, u}"] & \St(\Phi, R). \end{tikzcd}\]
\end{lemma}
\begin{proof}
Observe that the automorphism $\chi_{\omega, (u; u)}$ of $\St(\Phi, D_{R, I})$ preserves subgroups
$\Ker(p_i^*)$, $i=1, 2$ and hence their commutator subgroup $C$.
The required automorphism $\widetilde{\chi}_{\omega, u}$ now can be obtained by restricting $\chi_{\omega, (u; u)}$ to $\Ker(p_1^*)$.
The commutativity of the diagram is obvious.
\end{proof}

In the sequel we will use the following formulae describing the action of $\chi_{\omega, u}$ on the elements $w_\alpha(u)$, $h_\alpha(u)$:
\begin{align}
\label{eq:chi-w} \chi_{\omega, u}\left(w_\alpha(v)\right) &= w_\alpha(u^{\langle \omega, \alpha \rangle} \cdot v), \\
\label{eq:chi-h} \chi_{\omega, u} (h_\alpha(v)) &= h_\alpha(u^{\langle \omega, \alpha \rangle} \cdot v) \cdot h_\alpha(u^{\langle \omega, \alpha\rangle})^{-1} = \{u^{\langle \omega, \alpha\rangle}, v\} \cdot h_\alpha(v).
\end{align}

The following lemma is analogous to~\cite[Lemma~3.1(c)]{Tu83}.
\begin{lemma} \label{lem:winv-chiw}
For any $w \in \StW(\Phi, A[X^{\pm 1}])$ the element $w^{-1} \cdot \chi_{\omega, X}(w)$ belongs to $\StH(\Phi, A[X^{\pm 1}])$.
\end{lemma}
\begin{proof}
Since $\StH(\Phi, A[X^{\pm 1}])$ is a normal subgroup of $\StW(\Phi, A[X^{\pm 1}])$, it suffices to verify the assertion for $w = w_\alpha(u)$.
Set $h = w^{-1} \cdot \chi_{\omega, X}(w)$.
Notice that
\begin{multline*} w_\alpha(1) \cdot h\cdot  w_\alpha(-1) = w_\alpha(-1)^{-1} \cdot w_\alpha(u)^{-1} \cdot w_{\alpha}(X^{\langle \omega, \alpha \rangle} u) \cdot w_\alpha(-1) = \\
= h_\alpha(u)^{-1} \cdot h_\alpha(X^{ \langle \omega, \alpha \rangle }u) = \{ X^{\langle \omega, \alpha \rangle}, u \} \cdot h_\alpha(X^{\langle \omega, \alpha \rangle}) \in \StH(\Phi, A[X^{\pm 1}]).\end{multline*}
Thus, we get from\cite[Lemme~5.2]{Ma69} and~\eqref{eq:steinberg},\eqref{eq:symbol-properties} that
\begin{multline} \label{eq:w-computation} h = \{ X^{\langle \omega, \alpha \rangle}, u \}{}^{w_\alpha(-1)}h_\alpha(X^{\langle \omega, \alpha \rangle}) = \\
= \{ X^{\langle \omega, \alpha \rangle}, u \} \{ -1, X^{\langle \omega, \alpha \rangle} \} h_{-\alpha}(X^{\langle \omega, \alpha \rangle}) = \\
= \{ X^{\langle \omega, \alpha \rangle}, -u \} \cdot h_\alpha^{-1}(X^{\langle \omega, \alpha \rangle}).\qedhere\end{multline}
\end{proof}

\begin{example} \label{exm:chi-linear}
Set $R = A[X^{\pm 1}]$.
Consider the following weights of the root system of $\rA_\ell$ (cf.~\cite[\S~VI.4.7]{Bou}):
\[ \varpi_{i} = \varepsilon_1 + \ldots + \varepsilon_i - \frac{i}{\ell+1}(\varepsilon_1 + \ldots + \varepsilon_n), \ 1 \leq i \leq \ell. \]
For $1\leq k\leq \ell+1$ and $u \in R^\times$ denote by $d_k(u)$ the matrix from $\GL(\ell+1, R)$ which differs from the unit matrix only in that it has the element $u$ on the $k$-th position of its diagonal.
Recall from~\cite[Corollary~4]{Ka77} that for any $g \in \GL(\ell+1, R)$ there exists an automorphism $\beta_g$ of $\St(\ell+1, R)$ ''modeling`` the automorphism $(-)^g \colon \GL(\ell+1, R) \to \GL(\ell+1, R)$ of right conjugation by $g$, i.\,e. such that $\pi \beta_g = (-)^g \pi$.

It is clear that in the linear case the automorphism $\chi_{-\varpi_1, u}$ (resp. $\chi_{\varpi_{\ell}, u}$) coincides with $\beta_{d_1(u)}$ (resp. $\beta_{d_{\ell+1}(u)}$).
For other Chevalley groups the maps $\chi_{\omega, u}$ model automorphisms of inner conjugation by weight elements $h_\omega(u)$ in the sense of~\cite[\S~4]{Vav09}.
\end{example}

Let $\omega \in P(\Phi)$ be a weight of $\Phi$.
Consider the subset $\mathcal{X}_\omega$ of the generating set of $\St(\Phi, A[X])$ consisting of those generators $x_{\alpha}(a) \in \St(\Phi, A[X])$ for which
$\langle \omega, \alpha\rangle < 0$ implies that $a \in A[X]$ is divisible by $X^{-\langle \omega, \alpha\rangle}$.
\begin{rem}
Notice that according to this condition if $\langle \omega, \alpha\rangle \geq 0$ then $\mathcal{X}_\omega$ contains $x_\alpha(a)$ for all $a \in A[X]$.
\end{rem}
Denote by $N_{\omega, \Phi}$ the subgroup of $\St(\Phi, A[X])$ generated by $\mathcal{X}_\omega$.
When $\Phi$ is clear from context we shorten the notation for this subgroup to just $N_\omega$.

\begin{dfn} \label{dfn:delta-pair}
Let $\omega \in P(\Phi)$ be a weight.
By definition, an {\it $\omega$-pair} is a pair of mutually inverse group homomorphisms
$\xymatrix{ \sigma(\omega)\colon N_\omega \ar[r] & \ar@<-1.0ex>[l] N_{-\omega}\colon \sigma(-\omega) }$ satisfying the following identity:
\begin{equation} \label{eq:sigmadef}
\sigma(\pm \omega)(x_\alpha(f)) = x_\alpha(X^{\pm\langle\omega, \alpha\rangle}\cdot f),
\text{ for all } x_\alpha(f) \in \mathcal{X}_{\pm\omega}.
\end{equation}\end{dfn}
It is clear that the maps $\sigma(\omega)$, $\sigma(-\omega)$ are uniquely determined by~\eqref{eq:sigmadef}, so at most one $\omega$-pair may exist for any given $\omega$.
Moreover, $\sigma(\omega), \sigma(-\omega)$ always make the following diagram commute:
\begin{equation} \label{eq:sigma-diagram}
\xymatrix{ N_\omega \ar[r]_{\sigma(\omega)}\ar@{^{(}->}[d] \ar@/^1.5pc/[rr]^{\mathrm{id}} & N_{-\omega} \ar@{^{(}->}[d] \ar[r]_{\sigma(-\omega)} & N_\omega \ar@{^{(}->}[d] \\
\St(\Phi, A[X]) \ar[d] & \St(\Phi, A[X]) \ar[d] & \St(\Phi, A[X]) \ar[d] \\
\St(\Phi, A[X^{\pm 1}]) \ar@<-0.0ex>[r]_{\chi_{\omega, X}} \ar@/_1.5pc/[rr]^{\mathrm{id}} & \St(\Phi, A[X^{\pm 1}]) \ar@<-0.0ex>[r]_{\chi_{-\omega, X}} & \St(\Phi, A[X^{\pm 1}]).} \end{equation}
The question of existence of an $\omega$-pair is rather complicated since $N_\omega$ itself is \textit{not} presented by generators and relations
(it is merely a subgroup of such group) and $\sigma(\pm \omega)$ can \textit{not} be extended to all of $\St(\Phi, A[X])$.
In the case when $\omega$ is a \textit{microweight} in order to show the existence of $\sigma(\omega)$ in~\cref{subsec:construction-sigma} one can use
the explicit presentation of the \textit{relative} Steinberg group $\St(\Phi, A[X], XA[X])$
(which in this case forms an essential part of the group $N_\omega$).
We will recall such explicit presentation in the next subsection.

Now we describe a construction which allows one to construct many $\omega$-pairs once one such pair is constructed.
Suppose that we already know that an $\omega$-pair exists.
We claim that $w \cdot \omega$-pair also exists for any $w \in W(\Phi).$
Indeed, for $g \in N_\omega$, $w \in \StW(\Phi, A)$ we define $\sigma(\overline{w} \cdot \omega)$ via the following identity:
\begin{equation} \label{eq:sigma-gen} \sigma(\overline{w} \cdot \omega)(g) = w \cdot \sigma(\omega)(w^{-1} g w) \cdot w^{-1}.\end{equation}
The correctness of this definition follows from the following
\begin{lemma} \label{lem:delta-weyl}
For $g \in N_{\overline{w} \cdot \omega}$, $w \in \StW(\Phi, A)$ one has $w^{-1} g w \in N_\omega$.
The expression in the right-hand side of~\eqref{eq:sigma-gen} depends only on the class $\overline{w} \in W(\Phi)$.
\end{lemma}
\begin{proof}
To show the existence of $\sigma(\overline{w}\cdot \omega)$ it is enough to consider the case $w = w_\alpha(u)$, $g = x_\beta(X^n v) \in \mathcal{X}_{\overline{w} \cdot \omega}$, for some $u \in A^\times$, $v \in A[X]$, which we may assume to be not divisible by $X$.
Clearly, $\overline{w_\alpha(u)} = s_\alpha$ and by our assumptions $n \geq -\min(0, \langle s_\alpha(\omega), \beta \rangle)$.
Notice that by~\cite[Lemme~5.1(b)]{Ma69}
\[ g^w = x_\beta(X^n v)^{w_\alpha(u)} = x_{s_\alpha(\beta)}(\eta_{\alpha, s_\alpha(\beta)} X^{n} \cdot v \cdot u^{\langle s_\alpha(\beta), \alpha \rangle}),\]
where $\eta_{\alpha, s_\alpha(\beta)} = \pm 1$ are the coefficients defined in~\cite[Lemme~5.1]{Ma69}.
From the orthogonality of $s_\alpha$ and the fact that $\Phi$ is simply-laced we also obtain that
$\langle s_\alpha(\omega), \beta \rangle = \langle \omega, s_\alpha(\beta) \rangle$, consequently, we conclude that $g^w \in \mathcal{X}_\omega$.
Thus, the first assertion of the lemma is proved.

Direct computation using the identities from~\cite[Lemma~5.1]{Ma69} also shows that
\begin{multline*}
{}^w\sigma(w)(g^w) = {}^w\sigma(w)(x_{s_\alpha(\beta)}(\eta_{\alpha, s_\alpha(\beta)} X^{n} \cdot v \cdot u^{\langle s_\alpha(\beta), \alpha \rangle})) = \\
= {}^w x_{s_\alpha(\beta)}(\eta_{\alpha, s_\alpha(\beta)} \cdot X^{n + \langle \omega, s_\alpha(\beta) \rangle} \cdot v \cdot u^{\langle s_\alpha(\beta), \alpha \rangle}) = \\
= x_{\beta}(\eta_{\alpha, s_\alpha(\beta)}^2 \cdot X^{n + \langle s_\alpha(\omega), \beta \rangle} \cdot v \cdot u^{\langle s_\alpha(\beta), \alpha \rangle - \langle s_\alpha(\beta), \alpha \rangle }) = x_{\beta}(X^{n + \langle s_\alpha(\omega), \beta \rangle} \cdot v).
\end{multline*}
It remains to notice that the expression in the right-hand side of the above formula does not depend on $u$ and agrees with~\eqref{eq:sigmadef},
which implies the second assertion. \end{proof}

\subsection{Another presentation of the relative linear Steinberg group} \label{subsec:rel-presentation}
The aim of this subsection is twofold.
First of all, we recall the presentation of the relative linear Steinberg group $\St(\rA_{n-1}, R, I) = \St(n, R, I)$ from~\cite{LS17}.
This presentation is formulated in terms of rows and columns rather than root unipotents and is similar to
M.~Tulenbaev's presentation of the relative linear Steinberg group from~\cite[Definition~1.5]{Tu83}.
The key advantage of this presentation over Tulenbaev's is that unlike the latter it applies to the Steinberg group $\St(4, R, I)$,
which will be important in the sequel.

Our second main goal is to define certain elements $X^d(u, v)$, $Y^d(v, u)$ which will help us
with the construction of morphisms $\sigma(\omega)$ from~\cref{dfn:delta-pair}.
These elements are inspired by Tulenbaev's elements $X_{v, Xw}^{\delta_k}$ defined on p.~145 of~\cite{Tu83} with
the important difference that our definitions are also adapted to the situation $n=4$.

For a column $u \in R^n$ we denote by $u^t$ its transpose (which is a row of length $n$).
We denote by $\Um(n, R) \subseteq R^n$ the subset consisting of unimodular columns,
i.\,e. columns $u \in R^n$ such that $v^t u = 1$ for some $v \in R^n$.
We denote by $e_1, \ldots e_n$ the standard basis of $R^n$.
We also denote by $e$ the identity matrix of size $n$.
If $u, v \in R^n$ is a pair of orthogonal columns then $t(u, v) = e + uv^t$ is an invertible matrix of size $n.$
For an element $g \in \GL(n, R)$ we denote by $g^*$ the transpose-inverse of $g$, i.\,e. $g^* = {g^{t}}^{-1}.$

\begin{prop}
\label{prop:rel-presentation}
Assume that $I$ is a splitting ideal of a commutative ring $R$.
Then for any $\ell\geq 3$ the group $\St(\rA_\ell,\,R,\,I)$ can be presented by means of two families of generators $F(u,\,v)$, $S(v,\,u)$
(where $u\in \E(n,\,R)e_1,$ and $v\in I^n$ are such that $u^{t}v=0$) subject to the following relations:
\begin{align}
&F(u,\,v)F(u,\,w)=F(u,\,v+w), \label{add4}\\
&S(u,\,v)S(w,\,v)=S(u+w,\,v), \label{add5}\\
&F(u,\,v)F(u',\,v')F(u,\,v)^{-1}=F(t(u,\,v)u',\,t(v,\,u)^{-1} v'), \label{conj3} \\
&F(me_1,\,m^{*}e_{2}a)=S(me_{1}a,\,m^{*}e_{2}),\ \text{for all $a\in I$,}\, m \in \E(n, R). \label{coef-move}
\end{align}
\end{prop}
\begin{proof}
This is precisely~\cite[Proposition 3.10]{LS17} combined with~\cite[Proposition~8]{S15}.
\end{proof}

In what follows we assume $n \geq 4$.
We denote by $D(u)$ the subset of $R^n$ consisting of all columns $v$ which are orthogonal to $u$
(i.\,e. $u^{t} v = 0$) and have at least two zero entries.
Recall from 3.2 of~\cite{Ka77} that for every $u, v, w \in R^n$ such that $u^t v = 0$ there
is a decomposition of $(w^t u) \cdot v$ into a sum of elements of $D(u)$, called \textit{canonical decomposition}:
\begin{equation}
\label{eq:canonical} (w^tu) \cdot v=\sum_{i<j}u_{ij} c_{ij}(v, w),
\end{equation}
where $u_{ij} =e_i u_j-e_j u_i \in D(u)$ and $c_{ij}(v, w) =v_i w_j-v_j w_i \in R$.

\begin{lemma}
\label{lem:xsmall-properties}
Let $v, w \in R^n$ be such that $v^t w = 0$ and assume, moreover, that either $v$ or $w$ has at least one zero entry.
Under these assumptions one can define certain element $x(v, w) \in \St(n, R)$ such that $\pi(x(v, w)) = t(v, w) = e + vw^t$.
The elements $x(v, w)$ satisfy the following properties:
\begin{lemlist}
\item \label{itm:xsmall-scalar} If $v$ or $w$ has at least two zero entries, then $x(v, wa) = x(va, w)$ for $a\in R$.
\item \label{itm:xsmall-additivity} If $w_1$ and $w_2$ have at least two zero entries of which at least one entry is common
then $x(v, w_1) \cdot x(v, w_2) = x(v, w_1+w_2)$ and $x(w_1, v) \cdot x(w_2, v) = x(w_1 + w_2, v)$.
\item \label{itm:xsmall-commute} If $v$, $v'$ are simultaneously orthogonal to $w$ and $w'$, and the elements $w, w'$ both have at least two zero entries then
$[x(v, w),\ x(v', w')] = 1$.
\item \label{itm:xsmall-conj} If $g = x_{ij}(a)$ is a generator of the Steinberg group and $v$ or $w$ has at least two zero entries then
$g \cdot x(v, w) \cdot g^{-1} = x(gv, g^*w)$.
\end{lemlist}
\end{lemma}
\begin{proof}
See~\cite[Lemma~1.1]{Tu83}.
\end{proof}

Now assume that $I$ is an ideal of a ring $R$.
We also assume that $R$ itself is a subring of a larger ring $S$.
We now define two families of elements of $\St(n, R)$.
\begin{dfn} \label{dfn:xy-def}
Let $u \in \Um(n, R)$ and $v \in I^n$ be a vector such that $u^{t}v = 0$.
Let $v = \sum_r v^r$ be some decomposition of $v$ into a finite sum of elements $v^r \in D(u) \cap I^n$.
Under our assumptions on $u$ and $v$ such decomposition always exists, see~\eqref{eq:canonical}.

Let $d = \mathrm{diag}(d_1, \ldots, d_n)$ be an element of the subgroup $T(n, S)$ of diagonal matrices of $\GL(n, S)$ such that
$d^{-1} \cdot u \in R^n,\ d \cdot I^n \subseteq R^n.$
Under these assumptions we set
\begin{equation} \label{eq:X-def}
X^d(u, v) \coloneqq \prod_i x(d^{-1}u, dv^r),
\end{equation}

Not let $d'$ be an element of $T(n, S)$ such that $d' u\in R^n,\ {d'}^{-1} \cdot I^n \subseteq R^n$.
Under these assumptions we set
\begin{equation} \label{eq:Y-def}
Y^{d'}(v, u) \coloneqq \prod_i x({d'}^{-1} v^r, {d'}u).
\end{equation}
\end{dfn}

\begin{lemma}
\label{lem:xy-wd}
The elements $X^d(u, v)$, $Y^{d'}(v, u)$ are well-defined, i.\,e. they do not depend on the choice of decomposition for $v$.
\end{lemma}
\begin{proof}
Let $v = \sum_r v^r$ be a decomposition as above.
Since $u$ is unimodular there exists $w$ such that $w^t u = 1$.
Since each $v^r$ is orthogonal to $u$ we can write the canonical decomposition
$v^r = (w^tu)v^r = \sum_{i<j} u_{ij} c_{ij}(v^r, w)$, moreover $\sum_{r} c_{ij}(v^r, w) = c_{ij}(v, w)$.
Now using~\cref{lem:xsmall-properties} we obtain that
\begin{multline*}
X^d(u, v) =
\prod\limits_r x(d^{-1}u, dv^r) = \prod\limits_{r}\prod\limits_{i<j} x(d^{-1} u, du_{ij}c_{ij}(v^r, w)) = \\ =
\prod\limits_{i<j} x(d^{-1} u, d u_{ij}c_{ij}(v, w)). \qedhere
\end{multline*}
\end{proof}

\begin{lemma}
\label{lem:xy-conj} Suppose that $g = x_{hk}(a)$ is a generator of $\St(n, R)$ such that $m = d\pi(g)d^{-1} \in \E(n, R)$, then
\begin{equation*}
g \cdot X^d(u, v) \cdot g^{-1} = X^d(mu, m^*v) \text{ and } g \cdot Y^d(v, u) \cdot g^{-1} = Y^d(mv, m^*u).
\end{equation*}
\end{lemma}
\begin{proof}
Choose $w\in R^n$ such that $w^t u = 1$ and write $v = \sum_{1\leq i<j\leq n} u_{ij} c_{ij}$, where $u_{ij}$ and $c_{ij} = c_{ij}(v, w)$ are as in~\eqref{eq:canonical}.
By~\cref{lem:xsmall-properties} we can write $g' = g \cdot X^d(u, v) \cdot g^{-1}$ as follows:
\begin{equation} \nonumber
\prod\limits_{1\leq i<j\leq n} x(\pi(g) \cdot d^{-1} \cdot u, \pi(g)^* \cdot d \cdot u_{ij}c_{ij})
= \prod\limits_{1\leq i<j\leq n} x( d^{-1} \cdot m \cdot u, d \cdot m^* \cdot u_{ij}c_{ij}).
\end{equation}
Now for every factor in the right-hand side we do the following:
if $\{i, j\} = \{h, k\}$ or $\{i, j\} \cap \{h, k\} = \emptyset$ we leave the factor unchanged,
otherwise, if, $|\{i, j\} \cap \{h, k\}| = 1$, say, $j = h$, $i\neq k$, we further decompose it as follows:
\begin{equation} \nonumber
x(d^{-1} \cdot u', d \cdot m^* \cdot u_{ij} c_{ij}) =
x(d^{-1} \cdot u', d \cdot u'_{ij} c_{ij}) \cdot
x(d^{-1} \cdot u', d \cdot u'_{ki} bc_{ij}),
\end{equation}
where $u' = m \cdot u = t_{jk}(b) \cdot u$ and $b = d_j \cdot a \cdot d_k^{-1} \in R$.
Notice that
\begin{multline*}
m^* \cdot u_{ij} = t_{kj}(-b) \cdot u_{ij} = t_{kj}(-b) \cdot (e_i u_j-e_j u_i) = e_iu_j - e_ju_i + e_k b u_i =
\\ = \left(e_i (u_kb + u_j) - e_j u_i\right) + (e_k u_i-e_i u_k)b = (e_i u'_j-e_j u'_i) + (e_k u'_i-e_i u'_k)b = u'_{ij}+u'_{ki}b.
\end{multline*}

Thus, $g' = \prod_{s} x(d^{-1} \cdot u', d \cdot v'_s) $, where $v'_s \in D(u')$ and $\sum v'_s = m^*v$.
It is clear now that $g' = X^d(m u, m^* v)$ by~\cref{dfn:xy-def}.
The argument for the generator $Y^d(v, u)$ is similar.
\end{proof}

\section{Curtis--Tits type presentations} \label{sec:affine}
Throughout this section $A$ denotes an arbitrary commutative ring.

\subsection{Curtis--Tits type presentation of affine Steinberg groups} \label{subsec:curtis-tits}
In this subsection we briefly recall the theory of Steinberg groups in the Kac--Moody setting.
Only in this subsection we allow root systems to be infinite.

Recall that to any generalized Cartan matrix (or, for short, GCM) $C$ one can associate a (possibly infinite) root system $\Phi = \Phi_C$.
In this general setting, the notion of the Weyl group $W(\Phi)$ is introduced.
Additionally, there are analogues of the concepts of simple, positive and negative roots, as discussed in e.\,g.~\cite[\S~16]{Ca05}.
Recall that a root $\alpha \in \Phi$ is called \textit{real} if it lies in the orbit of some simple root $\alpha_i \in \Pi$ under the action of the Weyl group $W(\Phi)$, cf.~\cite[\S~16.3]{Ca05}.
Recall that a pair of real roots $\alpha, \beta$ is called \textit{prenilpotent} if there exists an element $w$ of the Weyl group $W(\Phi)$ such that both $w\cdot \alpha$ and $w \cdot \beta$ belong to $\Phi^+,$ cf.~\cite[\S~3]{A16}
A consequence of this condition is that $(\mathbb{Z}_{>0} \alpha + \mathbb{Z}_{>0}\beta)\cap \Phi$ consists of real roots.

For our purposes it suffices to focus on a specific subclass of GCMs of \textit{affine} type.
Let $\Phi$ be a finite irreducible root system of rank at least $2$.
Following \S~4 of~\cite{A16} we denote by $\widetilde{\Phi}$ the corresponding unfolded affine root system (sometimes it is denoted by $\Phi^{(1)}$, cf.~\cite[Table~2]{A16}).
The Dynkin diagram of $\widetilde{\Phi}$ is the \textit{extended} Dynkin diagram of $\Phi$, cf.~\cite[\S~17.1]{Ca05}.
Such diagrams for $\Phi$ of type $\rD_\ell$, $\rE_6$ and $\rE_7$ are shown on Figure~1.
The vertex corresponding to the affinizing simple root is marked with yellow color and index $0$ (other markings on the diagram will be explained in the subsequent text).

\tikzset{
root/.style={circle, draw, minimum size=0.2cm, inner sep=0},
zeroroot/.style={circle, draw, minimum size=0.2cm, inner sep=0, fill=yellow},
highlighted/.style={circle, draw, minimum size=0.2cm, inner sep=0, fill=green},
levi/.style={draw, dashed, rounded corners},
dottededge/.style={dotted},
labeled/.style={below}
}
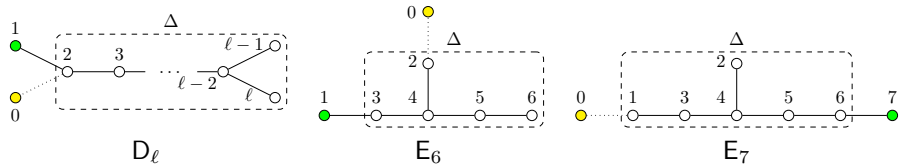
\begin{figure}[hb]\label{fig:dynkin-diagrams}
\begin{longtable}{ c c c }
\scalebox{0.69}{\begin{tikzpicture}
\node[root, highlighted, label=$1$] (d1) at (0,0.5) {};
\node[root, label=$2$] (d2) at (1,0) {};
\node[root, label=$3$] (d3) at (2,0) {};
\node at (3,0) {\ldots};
\node[root, label={[xshift=-0.5cm, yshift=-0.6cm]$\ell-2$}] (dl2) at (4,0) {};
\node[root, label={[xshift=-0.6cm, yshift=-0.35cm]$\ell-1$}] (dl1) at (5,0.5) {};
\node[root, label={[xshift=-0.5cm, yshift=-0.3cm]$\ell$}] (dl) at (5,-0.5) {};
\node[root, zeroroot, label=below:{$0$}] (d0) at (0,-0.5) {};

\draw (d1) -- (d2) -- (d3) -- (2.5,0);
\draw (3.5,0) -- (dl2) -- (dl1);
\draw (dl2) -- (dl);
\draw[dottededge] (d2) -- (d0);

\begin{scope}
\node[levi, fit=(d2) (d3) (dl) (dl1) (dl2), label=above:{$\Delta$}] {};
\end{scope}
\end{tikzpicture}}
&
\scalebox{0.69}{\begin{tikzpicture}
\begin{scope}
\node[root, highlighted, label=$1$] (e1) at (0,0) {};
\node[root, label=$3$] (e3) at (1,0) {};
\node[root, label={[xshift=-0.3cm]$4$}] (e4) at (2,0) {};
\node[root, label=$5$] (e5) at (3,0) {};
\node[root, label=$6$] (e6) at (4,0) {};
\node[root, label={[xshift=-0.3cm,yshift=-0.3cm]$2$}] (e2) at (2,1) {};
\node[root, zeroroot, label=left:$0$] (e0) at (2,2) {};
\draw (e1) -- (e3) -- (e4) -- (e5) -- (e6);
\draw (e2) -- (e4);
\draw[dottededge] (e2) -- (e0);

\begin{scope}
\node[levi, fit= (e2) (e3) (e4) (e5) (e6), label=above:{$\Delta$}] {};
\end{scope}
\end{scope}
\end{tikzpicture}}
&
\scalebox{0.69}{\begin{tikzpicture}
\begin{scope}
\node[root, zeroroot, label={0}] (e0) at (0,0) {};
\node[root, label={$1$}] (e1) at (1,0) {};
\node[root, label={$3$}] (e3) at (2,0) {};
\node[root, label={[xshift=-0.3cm]$4$}] (e4) at (3,0) {};
\node[root, label={$5$}] (e5) at (4,0) {};
\node[root, label={$6$}] (e6) at (5,0) {};
\node[root, label={[xshift=-0.3cm,yshift=-0.3cm]$2$}] (e2) at (3,1) {};
\node[root, highlighted, label={$7$}] (e7) at (6,0) {};

\draw (e1) -- (e3) -- (e4) -- (e5) -- (e6) -- (e7);
\draw (e2) -- (e4);
\draw[dottededge] (e0) -- (e1);

\begin{scope}
\node[levi, fit=(e1) (e2) (e3) (e4) (e5) (e6), label=above:{$\Delta$}] {};
\end{scope}
\end{scope}
\end{tikzpicture}} \\
\text{$\rD_\ell$} &
\text{$\rE_6$} &
\text{$\rE_7$}
\end{longtable}
\caption{Root markings on extended Dynkin diagrams}
\end{figure}

Now we are ready to recall the presentation of the Steinberg group functor $\St(C, -)$ in the Kac--Moody setting.
In fact, two such presentations have been proposed: the definition of J.~Tits~\cite[\S~3.6]{Ti87} and the definition of J.~Morita and U.~Rehmann~\cite[\S~2]{MR90}.
These definitions agree if the Dynkin diagram of $C$ does not contain connected components of type $\rA_1$ (as is the case in our setting), see~\cite[\S~3]{A16} or~\cite[\S~6]{A13}.

The set of generators for $\St(C, A)$ consists of all $x_\alpha(a)$, where $a \in A$ and $\alpha$ ranges over the set $\Phi^\mathrm{re}$ of real roots of $\Phi$.
The set of defining relations for $\St(C, A)$ consists of the following two families of relations:
\begin{align}
x_\alpha(a) x_\alpha(b) & =  x_\alpha(a + b), \label{SKM1} \\
[x_\alpha(a), x_\beta(b)] & =  \prod \limits_{\substack{i\alpha+j\beta \in \Phi^\mathrm{re} \\ i, j > 0 }} x_{i\alpha+j\beta}(N_{\alpha, \beta}^{i, j} \cdot a^i b^j), \label{SKM2}
\end{align}
where $a, b \in A$ and $\alpha, \beta$ range over all {\it prenilpotent} pairs of real roots of $\Phi$.

\begin{lemma} \label{lem:affine-vs-loop} For every irreducible finite root system $\Phi$ of rank $\geq 2$ one has $\St(\widetilde{\Phi}, A) \cong \St(\Phi, A[X, X\inv])$.
\end{lemma}
\begin{proof}
Notice that the set of real roots of $\widetilde{\Phi}$ is isomorphic to $\Phi \times \ZZ$, see~\cite[\S~3]{A16}.
Also recall from~\cite[\S~5]{A16} that a pair $(\alpha, n), (\beta, m)$ is prenilpotent iff $\alpha \neq - \beta$.

This description allows one to unfold the defining relations~\eqref{SKM1}--\eqref{SKM2}.
For example, in the simply-laced case one obtains the following list of defining relations for $\St(\widetilde{\Phi}, A)$:
\begin{align}
x_{(\alpha, m)}(a)\cdot x_{(\alpha, m)}(b)&=x_{(\alpha, m)}(a+b),  \label{AR1}\\
[x_{(\alpha, m)}(a),\,x_{(\beta, n)}(b)]  &=x_{(\alpha+\beta, n+m)}(N_{\alpha,\beta} \cdot ab),\text{ for }\alpha+\beta\in\Phi, \label{AR2} \\
[x_{(\alpha, m)}(a),\,x_{(\beta, n)}(b)]  &=1,\text{ for }\alpha+\beta\not\in\Phi\cup0. \label{AR3}
\end{align}

It remains to notice that the homomorphism given by $x_{(\alpha, m)}(a) \mapsto x_\alpha(aX^m)$ is an isomorphism with the inverse given by
$x_\alpha(a_{n}X^n + \ldots + a_m X^m) \mapsto \prod_{i=n}^m x_{(\alpha, i)}(a_i)$, $a_i \in A$, $n \leq m$, $n, m\in \ZZ$, cf. e.\,g.~\cite[\S~5.1]{LS20}.
\end{proof}

Our next goal is to formulate the so-called \textit{Curtis--Tits presentation} of affine Steinberg groups discovered by D.~Allcock in~\cite{A16, A13}.
This presentation has the advantage of being formulated in terms of the Dynkin diagram of $\Phi$ rather than the (possibly infinite) set of real roots of $\Phi$.
Its other advantage is that it requires no choice of structure constants in its statement.
Allcock's result is a generalization of Curtis--Tits presentation of the Steinberg group of a finite root system.
Results of such type have been known since 1960's, see e.\,g.~\cite[Theorem~B]{DS74}.

Recall that $\{ \alpha_1, \ldots \alpha_\ell \}$ is the set of simple roots of $\Phi$.
We denote by $\alpha_0$ the opposite root to the maximal root of $\Phi$ (i.\,e. $\alpha_0 := -\alpha_\mathrm{max}$).
Notice that $(\alpha_0, 1)$ corresponds to the added $0$ node of the extended Dynkin diagram of $\widetilde{\Phi}$ depicted on Figure 1, cf.~\cite[\S~4]{A16}.
We denote by $X_0(A)$ the root group $\{ X_0(a) \mid a \in A\}$ (as a group it is isomorphic to the additive group of $A$).
We also denote by $j$ the root adjacent to $0$ on the extended Dynkin diagram of $\Phi$.
\begin{prop} \label{prop:Allcock-affine} For a finite simply-laced irreducible root system $\Phi$ and an arbitrary commutative ring $A$
the group $\St(\Phi, A[X, X\inv])$ is isomorphic to the free product of $\St(\Phi, A)$, the group $X_0(A)$ and the infinite cyclic group $\langle S_0 \rangle$
amalgamated over the subgroup generated by the following list of relations:
{\allowdisplaybreaks\begin{align}
[S_0^2, X_0(a)] & = 1 & \text{ for $a \in A$; } \label{eq:Allcock-2} \\
X_0(1) \cdot {}^{S_0} X_0(1) \cdot X_0(1) & = S_0; \label{eq:Allcock-3} \\
[S_0, w_{\alpha_i}(1)] & = 1 &  \text{for $i$ unjoined with $0$;} \label{eq:Allcock-4} \\
[S_0, x_{\alpha_i}(a)] & = 1, &  \label{eq:Allcock-5-1}\\
[w_{\alpha_i}(1), X_0(a)] & = 1 & \text{for $a \in A$, $i$ unj. with $0$;} \label{eq:Allcock-5-2} \\
[X_0(a), x_{\alpha_i}(b)] & = 1 & \text{for $a, b \in A$, $i$ unj. with $0$;} \label{eq:Allcock-6} \\
S_0 \cdot w_{\alpha_j}(1) \cdot S_0 & = w_{\alpha_j}(1) \cdot S_0 \cdot w_{\alpha_j}(1); \label{eq:Allcock-7} \\
{}^{S_0^2} w_{\alpha_j}(1) & = w_{\alpha_j}(-1); \label{eq:Allcock-8-1} \\
{}^{w_{\alpha_j}^2(1)} S_0 & = S_0^{-1}; \label{eq:Allcock-8-2} \\
x_{\alpha_j}(a) \cdot S_0 \cdot w_{\alpha_j}(1) & = S_0 \cdot w_{\alpha_j}(1) \cdot X_0(a), & \label{eq:Allcock-9-1} \\
X_0(a) \cdot w_{\alpha_j}(1) \cdot S_0 & = w_{\alpha_j}(1) \cdot S_0 \cdot x_{\alpha_j}(a), & \label{eq:Allcock-9-2} \\
{}^{S_0^2} x_{\alpha_j}(a) & = x_{\alpha_j}(-a), & \label{eq:Allcock-10-1} \\
{}^{w_{\alpha_j}^2(1)} X_0(a) & = X_0(-a) & \text{for $a \in A;$} \label{eq:Allcock-10-2} \\
[X_0(a), {}^{S_0} x_{\alpha_j}(b)] &= 1, & \label{eq:Allcock-11-1} \\
[x_{\alpha_j}(a), {}^{w_{\alpha_j}(1)}X_0(b)] &= 1, & \label{eq:Allcock-11-2} \\
[X_0(a), x_{\alpha_j}(b)] &= {}^{S_0} x_{\alpha_j}(ab) & \text{for $a, b \in A.$} \label{eq:Allcock-12}
\end{align}}
\end{prop}
\begin{proof}
This is a direct consequence of our~\cref{lem:affine-vs-loop} and the presentation of~\cite[Theorem~1]{A16} (or, alternatively, \cite[Theorem~1.1]{A13} combined with~\cite[Theorem~1.3]{A13}).
The list of relations in the statement is obtained from \cite[Table~1]{A16} by substituting concrete values of $i, j$, identifying Allcock's $X_{i}(a)$ with
$x_{\alpha_i}(a)$ and $S_i$ with $w_{\alpha_i}(a)$ for all $i\neq 0$ and then omitting all the relations are already satisfied in $\St(\Phi, A)$ (i.\,e. the relations not involving $X_0(a)$ or $S_0$).
The only further simplification is that we have omitted the relation $[X_{\alpha_j}(b), X_0(a)] = {}^{w_{\alpha_j}(1)} X_0(ab)$ (the relation symmetric to~\eqref{eq:Allcock-12})
because it is a consequence of~\eqref{eq:Allcock-9-2} and~\eqref{eq:Allcock-12}.
\end{proof}

Our next goal is to slightly simplify the presentation of~\cref{prop:Allcock-affine} by using the larger group $\St(\Phi, A[X])$ instead of $\St(\Phi, A)$ in the statement.
\begin{cor} \label{cor:Allcock-simpler}
For a finite irreducible simply-laced root system $\Phi$ and an arbitrary commutative ring $A$ the Steinberg group $\St(\Phi, A[X, X\inv])$ is isomorphic
to the free product of $\St(\Phi, A[X])$ and the infinite cyclic group $\langle S \rangle$ amalgamated over the subgroup generated by the following list of relations:
{\allowdisplaybreaks\begin{align}
[S^2, x_{\alpha_0}(aX)] & = 1 & \text{ for $a \in A$; } \label{eq:simpler-2} \\
x_{\alpha_0}(X) \cdot {}^{S} x_{\alpha_0}(X) \cdot x_{\alpha_0}(X) & = S; \label{eq:simpler-3} \\
[S, w_{\alpha_i}(1)] & = 1 & \text{ for $i$ unjoined with $0$;} \label{eq:simpler-4} \\
[S, x_{\alpha_i}(a)] & = 1 & \text{ $i$ unj. with $0$, $a \in A$; } \label{eq:simpler-5-1}\\
S \cdot w_{\alpha_j}(1) \cdot S & = w_{\alpha_j}(1) \cdot S \cdot w_{\alpha_j}(1); \label{eq:simpler-7} \\
{}^{S^2} w_{\alpha_j}(1) & = w_{\alpha_j}(-1); \label{eq:simpler-8-1} \\
{}^{w_{\alpha_j}^2(1)} S & = S^{-1}; \label{eq:simpler-8-2} \\
x_{\alpha_j}(a) \cdot S \cdot w_{\alpha_j}(1) & = S \cdot w_{\alpha_j}(1) \cdot x_{\alpha_0}(aX), & \label{eq:simpler-9-1} \\
x_{\alpha_0}(aX) \cdot w_{\alpha_j}(1) \cdot S & = w_{\alpha_j}(1) \cdot S \cdot x_{\alpha_j}(a), & \label{eq:simpler-9-2} \\
{}^{S^2} x_{\alpha_j}(a) & = x_{\alpha_j}(-a) & \text{ for $a \in A$; } \label{eq:simpler-10-1} \\
[x_{\alpha_0}(aX), {}^{S} x_{\alpha_j}(b)] &= 1, & \label{eq:simpler-11-1} \\
[x_{\alpha_0}(aX), x_{\alpha_j}(b)] &= {}^{S} x_{\alpha_j}(ab) & \text{for $a, b \in A.$} \label{eq:simpler-12}
\end{align}}
\end{cor}
\begin{proof}
The list of relations is obtained from~\eqref{eq:Allcock-2}--\eqref{eq:Allcock-12} by identifying $X_0(a)$ with $x_{\alpha_0}(aX)$, $S_0$ with $S$ and omitting those relations
which are satisfied in the group $\St(\Phi, A[X])$ by virtue of Lemmas 5.1--5.2 of ~\cite{Ma69}.
\end{proof}
\begin{rem}
Notice that the generator $S$ in the above presentation corresponds to the element $w_{-\alpha_\mathrm{max}}(X) \in \St(\Phi, A[X, X\inv])$ under the above isomorphism.
\end{rem}

\subsection{Relationship between $\St(\Phi, A[X, X\inv])$ and $\St(\Phi, A[X])$.} \label{subsec:short-presentation}
Now let $\Phi$ be a root system of type $\rD_\ell$, $\rE_6$ or $\rE_7$.
Denote by $k$ the vertex of the extended Dynkin diagram marked green on Figure~1.
Recall that in each case $m_k(\alpha_\mathrm{\max}) = 1$, therefore $m_k(\alpha) = 1$ for all $\alpha \in \Sigma_k^+$ and the subgroup $\UU(\Sigma^+_k, R)$ is abelian.
Also the weight $\varpi_k$ is a \textit{microweight} of $\Phi$, cf. e.\,g.~\cite[\S~2]{Ge17}.

Denote by $G$ the amalgamated product of $\St(\Phi, A[X])$ and the cyclic group $\langle \sigma \rangle$ amalgamated over the subgroup generated by the following relations:
\begin{align}
{}^\sigma x_{\alpha}(f) = & x_{\alpha} (Xf), & \alpha \in \Sigma^+_k, f \in A[X], \label{eq:sigma-sigma-plus} \\
x_{\beta}(f)^ \sigma     =& x_{\beta} (Xf), & \beta \in \Sigma^-_k, f \in A[X], \label{eq:sigma-sigma-minus} \\
[\sigma,\, x_\gamma(f)]   =& 1, & \gamma \in \Delta_k, f \in A[X]. \label{eq:sigma-delta}
\end{align}
It is clear that the action of the generator $\sigma$ is chosen to mimick the action of the weight element $\chi_{\omega_k, X}$ from~\cref{subsec:weight-automorphisms}.

We denote by $i_+$ the canonical homomorphism $\St(\Phi, A[X]) \to G$ and by $h_+$ the canonical embedding $A[X] \to A[X, X\inv]$.

The main result of this subsection is the following
\begin{prop} \label{prop:rel-poly-Laurent}
For $\Phi$ and $G$ as above there exists a group homomorphism $\varphi$ making the diagram below commute:
\[\begin{tikzcd} & \St(\Phi, A[X, X\inv]) \arrow[rd, dashrightarrow, "\varphi"] & \\
\St(\Phi, A[X]) \arrow{ru}{h_+^*} \arrow{rr}{i_+} & & G.
\end{tikzcd}\]
\end{prop}
\begin{proof}
We will use the presentation of~\cref{cor:Allcock-simpler} for $\St(\Phi, A[X, X\inv])$.
We only need to specify the value of $\varphi$ on the generator $S$.
Observe that $-\alpha_0 = \alpha_{\max} \in \Sigma_k^+$, therefore the obvious candidate for the role of $\varphi(x_{-\alpha_0}(aX\inv))$ would be $x_{-\alpha_0}(a)^\sigma$, cf.~\eqref{eq:sigma-sigma-plus},
which motivates the following definition for $\varphi(S)$:
\[\varphi(S) := x_{\alpha_0}(X) \cdot x_{-\alpha_0}(-1)^\sigma \cdot x_{\alpha_0}(X).\]
From~\eqref{eq:w-definition},\eqref{eq:sigma-sigma-minus} it follows that $\varphi(S) = w_{\alpha_0}(1)^\sigma$.

We need to check that our definition is correct, namely that $\varphi$ respects the relations listed in~\cref{cor:Allcock-simpler}.
We claim that this follows from the fact that the substitution of $w_{\alpha_0}(1)^\sigma$ into $S$ turns every formula from \cref{cor:Allcock-simpler}
into a $\sigma$-conjugate of a valid relation in $\St(\Phi, A[X])$.

We will prove this claim for relations~~\eqref{eq:simpler-3}, \eqref{eq:simpler-5-1} and \eqref{eq:simpler-9-1}, for other relations it is similar but easier.
Let us show that $\varphi$ preserves~\eqref{eq:simpler-3}.
Using~\eqref{eq:sigma-sigma-minus} combined with~\cite[Lemma~5.1b]{Ma69} we obtain that
\begin{multline*}
\varphi(x_{\alpha_0}(X) \cdot {}^{S} x_{\alpha_0}(X) \cdot x_{\alpha_0}(X)) = x_{\alpha_0}(1)^\sigma \cdot {}^{w_{\alpha_0}(1)^\sigma} x_{\alpha_0}(1)^\sigma \cdot x_{\alpha_0}(1)^\sigma = \\
= \left( x_{\alpha_0}(1) \cdot {}^{w_{\alpha_0}(1)} x_{\alpha_0}(1) \cdot x_{\alpha_0}(1)\right)^\sigma = \left(x_{\alpha_0}(1) \cdot x_{-\alpha_0}(-1) \cdot x_{\alpha_0}(1)\right)^\sigma = w_{\alpha_0}(1)^\sigma = \varphi(S).
\end{multline*}

Let us show that $\varphi$ preserves~\eqref{eq:simpler-5-1}.
Set $g = x_{\alpha_i}(a)$ if $i \neq k$ and $g = x_{\alpha_k}(aX)$ if $i = k$.
From~\eqref{eq:sigma-sigma-plus},\eqref{eq:sigma-delta} and the commutator formulae we obtain that
\begin{equation*}
\varphi([S, x_{\alpha_i}(a)]) = [w_{\alpha_0}(1)^\sigma, x_{\alpha_i}(a)] = [w_{\alpha_0}(1)^\sigma, g^\sigma] = [w_{\alpha_0}(1), g]^\sigma = 1^\sigma = 1.
\end{equation*}

Let us show that $\varphi$ preserves~\eqref{eq:simpler-9-1}.
We claim that in $\St(\Phi, A)$ one has
\begin{equation} \label{eq:simpler-relation} {}^{w_{\alpha_0}(1) w_{\alpha_j}(1)} x_{\alpha_0}(a) = x_{\alpha_j}(a). \end{equation}
This identity can be verified by direct computation using~\cite[Lemma~5.1]{Ma69} and the standard identities for structure constants~\cite[\S~14]{VP} or,
alternatively, it can be obtained by applying the evaluation homomorphism
$\mathrm{ev}_{X=1}^*\colon \St(\Phi, A[X, X\inv]) \to \St(\Phi, A)$ to both sides of the relation~\eqref{eq:Allcock-9-1}.
Consequently, from~\eqref{eq:sigma-sigma-minus},\eqref{eq:sigma-delta} and~\eqref{eq:simpler-relation} we obtain that
\begin{multline*}
\varphi(x_{\alpha_j}(a) \cdot S \cdot w_{\alpha_j}(1)) = x_{\alpha_j}(a) \cdot w_{\alpha_0}(1)^\sigma \cdot w_{\alpha_j}(1) = \\
= \left(x_{\alpha_j}(a) \cdot w_{\alpha_0}(1) \cdot w_{\alpha_j}(1)\right)^\sigma = \left(w_{\alpha_0}(1) \cdot w_{\alpha_j}(1) \cdot x_{\alpha_0}(a)\right)^\sigma = \\
= w_{\alpha_0}(1)^\sigma \cdot w_{\alpha_j}(1) \cdot x_{\alpha_0}(aX) = \varphi(S \cdot w_{\alpha_j}(1) \cdot x_{\alpha_0}(aX)). \qedhere
\end{multline*}
\end{proof}

\subsection{Relative Curtis--Tits decompositions}
We aim of this subsection is to recall the amalgamation theorem for relative Steinberg groups from~\cite{S15}.

Notice that Allcock's presentation from~\cite{A16, A13} uses the identity element of $R$ in its statement,
so it is not possible to directly generalize it to the relative case by replacing $R$ with an ideal $I$.
A less naive idea to amalgamate rank $2$ relative Steinberg groups over the edges of the Dynkin diagram of $\Phi$ also would not work
because the resulting group would not even contain root subgroups $X_\alpha(I)$ for all $\alpha \in \Phi$.
Nevertheless, it turns out to be possible to obtain a weak analogue of Curtis--Tits decomposition for relative Steinberg groups of simply-laced type
by decomposing $\St(\Phi, R, I)$ into an amalgamated product of multiple copies of Steinberg groups $\St(\rA_3, R, I) \cong \St(4, R, I)$.

Let $\Phi$ be an arbitrary simply-laced irreducible root system of rank $\geq 3$.
We denote by $A_3(\Phi)$ the set consisting of all root subsystems of type $\rA_3$ of $\Phi$.
Under our assumption on $\Phi$ the relative Steinberg group $\St(\Phi, R, I)$ is generated by elements
$z_\alpha(m, a) = x_\alpha((m; 0))^{x_{-\alpha}((a; a))}C$, $\alpha \in \Phi$, $a \in R$, $m \in I$, see~\cite[\S~3.1]{S15}.

Denote by $\widetilde{G}$ the free product of relative Steinberg groups $\St(\Psi, R, I)$, where $\Psi \in A_3(\Phi)$.
For $\Psi \in A_3(\Phi)$ also denote by $i_\Psi$ the canonical embedding $\St(\Psi, R, I) \to \widetilde{G}$.
We denote by $G$ the quotient of $\widetilde{G}$ modulo all relations of the form $i_{\Psi_1}(z_\alpha(m, a)) = i_{\Psi_2}(z_\alpha(m, a))$,
where $\Psi_1, \Psi_2 \in A_3(\Phi)$ both contain a common root $\alpha \in \Phi$ and $a\in R$, $m \in I$.

\begin{thm}\label{thm:relPres} In the above notation the group $\St(\Phi, R, I)$ is isomorphic to $G$ as an abstract group. \end{thm}
\begin{proof}
See~\cite[\S~2]{LS20}.
\end{proof}

\begin{rem}
Recently E.~Voronetsky has found an explicit presentation of the relative Steinberg group $\St(\Phi, R, I)$
by means of generators and relations, see~\cite{V22}.
This result can be thought of as a further strengthening of~\cref{thm:relPres}
\end{rem}

\section{Horrocks theorem for $\K_2$} \label{sec:horrocks}
\begin{dfn}
Let $K$ be a functor from commutative rings to groups.
Recall from~\cite{LSV2} that $K$ is called \textit{locally acyclic (resp., locally acyclic for domains)} if for every commutative local ring (resp., domain)
$A$ the following diagram whose arrows are induced by natural embeddings is a pullback square:
\begin{equation}\label{eq:P1-square} \begin{tikzcd} K(A) \ar[r] \ar[d] & K(A[X]) \arrow{d} \\ K(A[X\inv]) \ar{r} & K(A[X, X\inv]). \end{tikzcd} \end{equation}
\end{dfn}
Observe that the top and left arrows in the above diagram are injective (they admit a section induced by evaluation at $X=1$).
Consequently, if~\eqref{eq:P1-square} is a pullback then all arrows on the diagram are injective.

Now we are ready formulate the $\K_2$-analogue of Horrocks theorem for domains,
which is the main result of this section.
\begin{thm}\label{thm:horrocks-k2}
Let $\Phi$ be a root system of type $\rA_{\geq 4}$, $\rD_{\geq 5}$, $\rE_6$ or $\rE_7$.
Then the functor $\K_{2}(\Phi, -)$ is locally acyclic for domains.
\end{thm}
Horrocks theorem for $\K_2$ has been previously known in the linear and even orthogonal case $\Phi=\rA_{\geq 4},\rD_{\geq 7}$, see~\cite[Proposition~4.3]{Tu83} and~\cite[Theorem~1]{LS20}, respectively.
Notice that the orthogonal Horrocks theorem for $\K_2$ was proved in~\cite{LS20} only under the additional assumption that $2$ is invertible.
Thus, the novelty of~\cref{thm:horrocks-k2} lies not only in its applicability to root systems of types $\rD_5$, $\rD_6$, $\rE_6$, and $\rE_7$, but also in the fact that
it is proved without assuming that 2 is invertible -- albeit under the requirement that the base ring is a domain.

Since our proof of \cref{thm:horrocks-k2} is rather long and technical, we will first outline its main structure.
Let \( A \) be a local ring with maximal ideal \( M \) and residue field \( \kappa \).
Recall that proofs of Horrocks' theorem for \( \K_1 \) are typically based on the decomposition of the elementary subgroup over the Laurent polynomial ring,
commonly referred to as \textit{Suslin's lemma} or \textit{Suslin's structure theorem}; see~\cite{Abe83, Su77}, \cite[\S~VI.6]{Lam10}.
In the linear case, this decomposition asserts that the elementary linear group
\( \E(n, A[X, X\inv]) = \Esc(\rA_{n-1}, A[X, X\inv]) \) admits the following group factorization for \( n \geq 3 \):
\begin{equation}\label{eq:triple-decomposition}
\E(n, A[X^{\pm 1}]) = \E(n, A[X]) \cdot B(A[X^{\pm 1}]) \cdot \E(n, A[X^{\pm 1}], M[X^{\pm 1}]).
\end{equation}
Here, \( \E(n, R, I) \) denotes the relative elementary subgroup, i.e., the kernel of the canonical reduction homomorphism \( \E(n, R) \to \E(n, R/I), \)
while \( B(R) \) denotes the Borel subgroup (i.e., the semidirect product of the unipotent radical \( \UU(\Phi^+, R) \) and the group of diagonal matrices).

To prove Suslin's structure theorem, one constructs a collection of pairwise commuting diagonal automorphisms \( \delta_i \) (\( 1 \leq i \leq n \))
and shows that the product \( V \) of the three subgroups on the right-hand side of~\eqref{eq:triple-decomposition} is stabilized by these automorphisms.
Since the product \( V \) is clearly stabilized by left translations by elements of \( \E(n, A[X]) \),
the decomposition~\eqref{eq:triple-decomposition} then follows from the fact that the minimal subgroup of \( \E(n, A[X, X\inv]) \)
that contains the image of \( \E(n, A[X]) \) and is simultaneously invariant under the action of all \( \delta_i \)
coincides with all of \( \E(n, A[X, X\inv]) \).

Tulenbaev's proof of Horrocks' theorem for \( \K_2 \) is analogous to Suslin's proof for \( \K_1 \).
The key difference, however, is that proving an analogue of the factorization~\eqref{eq:triple-decomposition} for the Steinberg group
\( \St(\Phi, A[X, X\inv]) \) alone is insufficient.
Instead, one must establish a much more refined result, namely to construct a ``model set'' \( \overline{V} \) for the group
\( \St(\Phi, A[X, X\inv]) \) using the same three ingredient groups: \( \St(\Phi, A[X]) \), the Borel subgroup, and the relative group
\( \St(\Phi, A[X^\pm], M[X^\pm]) \).
By a ``model set,'' we mean a set upon which \( \St(\Phi, A[X, X\inv]) \) acts simply transitively.
This technique likely originates from~\cite{ST76}, where it was employed to prove the injective stability theorem for \( \K_2 \).
It was subsequently adapted by M.~Tulenbaev in his proof of Horrocks' theorem for \( \K_2 \) in the linear case; see Propositions~4.1 and~4.3 in~\cite{Tu83}.
The same method is also used in~\cite[Theorem~3]{LS20}, a direct generalization of~\cite[Proposition~4.3]{Tu83}.
We adopt this technique to prove \cref{thm:horrocks-k2}, which should be regarded as a generalization of~\cite[Proposition~4.1]{Tu83}.

Our proof differs from Tulenbaev's in that it uses the economical presentation of the Steinberg group over Laurent polynomial rings, cf.~\cref{prop:rel-poly-Laurent}.
This allows us to construct a single ``weight automorphism'' $\delta_{\varpi_k}$ instead of multiple ones, as in Tulenbaev's approach.
Also, this streamlines the proof by avoiding the task of verifying that weight automorphisms commute with one another,
which is particularly hard to prove for root systems of types $\rD_\ell$ and $\rE_\ell$.

The key steps of the proof are as follows:
\begin{itemize}
\item \textbf{Step 1.} In~\cref{subsec:triples}, we introduce an abstract formalism of group decompositions, enabling us to construct the set \( \overline{V} \) modeling
the Steinberg group \( \St(\Phi, A[X, X\inv]) \). This serves as the groundwork for formulating the \( \K_2 \)-analogue of Suslin's structure theorem~\eqref{eq:triple-decomposition}.
\item \textbf{Step 2.} In~\cref{subsec:structure-theorem-overview}, we use this formalism to construct the "model set" \( \overline{V} \) and precisely formulate the \( \K_2 \)-analogue
of Suslin's structure theorem. We also demonstrate how \cref{thm:horrocks-k2} reduces to a concrete technical statement: the existence of an action of \( \St(\Phi, A[X, X\inv]) \) on \( \overline{V} \).
\end{itemize}

\cref{prop:rel-poly-Laurent} reduces the problem of constructing the action of \( \St(\Phi, A[X, X\inv]) \) on \( \overline{V} \) to the construction of
the automorphism \( \delta_{\varpi_k} \) modeling the action of the weight automorphism \( \chi_{\varpi_k, X} \) (cf.~\cref{subsec:weight-automorphisms}).
Here, the index \( k \) corresponds to the simple root marked in green in Figure~1 of \cref{subsec:curtis-tits}.
The construction of \( \delta_{\varpi_k} \) then proceeds as follows:
\begin{itemize}
\item \textbf{Step 3.} First, we establish the existence of a \( \varpi_k \)-pair in the sense of \cref{dfn:delta-pair}.
For the root system of type \( \rA_3 \), we construct the \( \varpi_1 \)-pair using the presentation from \cref{prop:rel-presentation}.
We then generalize this result using the amalgamation theorem for relative Steinberg groups (i.\,e. \cref{thm:relPres}), as shown in \cref{subsec:construction-sigma}.
\item \textbf{Step 4.} Finally, we complete the construction of the action of \( \St(\Phi, A[X, X\inv]) \) on \( \overline{V} \) and verify all the necessary relations.
This is carried out in \cref{sec:construction-delta}.
\end{itemize}

\subsection{Formalism of triples}\label{subsec:triples}
Let $\mu\colon M \to N$ and $\mu' \colon M' \to N'$ be a pair of precrossed modules (cf.~\cref{dfn:crossed-module}).
By definition, a \textit{homomorphism of precrossed modules} $(f, g)\colon \mu \to \mu'$ is a pair of group homomorphisms $f\colon M \to M'$, $g\colon N \to N'$ such that
\begin{enumerate}[ref=CH\arabic*, label=CH\arabic*)]
\item \label{ax:ch-cs} they complete $\mu$ and $\mu'$ to a commutative square, i.\,e. $\mu'f = g \mu$;
\item \label{ax:ch-ga} they preserve the group action, i.\,e. ${f(m)}^{g(n)} = f(m^n)$ for $n \in N$, $m \in M$.
\end{enumerate}

Now suppose that we are given the following cube-like commutative diagram of abstract groups:
\begin{equation} \label{eq:cube} \begin{split} \xymatrix{
G_{123} \ar[rr]_{f_{23}} \ar[dd]_{f_{12}} \ar[rd]_{f_{13}} &                        & G_{23} \ar@{-->}[dd]^(.3){g_2^3} \ar[rd]^{g^2_3} &           \\
& G_{13} \ar[rr]^(.3){g^1_3} \ar^(.3){g^3_1}[dd] &                   & G_3 \ar[dd]_{h_3} \\
G_{12} \ar@{-->}[rr]^(.3){g_2^1} \ar[rd]_{g_1^2}          &                        & G_2 \ar@{-->}[rd]_{h_2}         &           \\
& G_1 \ar[rr]_{h_1}              &                   & G.} \end{split} \end{equation}
Additionally, we make the following assumptions:
\begin{itemize}
\item $G_1$ acts on $G_{13}$ on the right; $G$ acts on $G_3$ on the right;
exponential notation is used for both actions, e.\,g. $h^g$ denotes the action of $g \in G$ on $h \in G_{3}$;
\item $g_1^3$ is a precrossed module with respect to the above action;
\item $h_3$ is a crossed module with respect to the above action;
\item $(g_3^1, h_1) \colon g_1^3 \to h_3$ is a homomorphism of precrossed modules.
\end{itemize}
Set $V = G_1 \times G_2 \times G_3$, $W = G_{12} \times G_{13} \times G_{23}$.
We define an operation $\star \colon W \times V \to V$ as follows.
For $v = (x, y, z) \in V$ and $w = (a, b, c) \in W$ we set
\[(a, b, c) \star (x, y, z) = (x \cdot g_1^3(b) \cdot g_1^2(a),\ g_2^1(a)^{-1} \cdot y \cdot g_2^3(c),\ g_3^2(c)^{-1} \cdot g_3^1(b)^{-h_2(y)} \cdot z).\]
Consider the set-theoretic map $h \colon V \to G$ given by $(x, y, z) \mapsto h_1(x) \cdot h_2(y) \cdot h_3(z)$.
It is clear from the definition of $\star$-operation that for $w \in W$ one has $h(w \star v) = h(v).$
Now let us define the relation associated with $\star$-operation.
We declare two elements $v, v' \in V$ congruent (denoted $v \sim_W v'$) if $v' = w \star v$ for some triple $w=(a, b, c) \in W$.
As the following lemma shows, this relation is an equivalence relation.
\begin{lemma} For every $v \in V$ one has
\begin{equation*}(a', b', c') \star \left( (a, b, c) \star v \right) = (a \cdot a', b \cdot {b'}^{g_1^2(a)^{-1}}, c \cdot c') \star v.\end{equation*}
\end{lemma}
\begin{proof}
Set $v'=(x', y', z') = (a, b, c) \star v$ and $(x'', y'', z'') = (a', b', c') \star v'$.
Since $g_1^3$ is a precrossed module we immediately obtain that
\begin{align*}
x'' =& x' \cdot g_1^3(b') \cdot g_1^2(a') = x \cdot g_1^3(b) \cdot g_1^2(a) \cdot g_1^3(b') \cdot g_1^2(a') = x \cdot g_1^3(b \cdot b'^{g_1^2(a)^{-1}}) g_1^2(a \cdot a'),\\
y'' =& g_2^1(a')^{-1} \cdot g_2^1(a)^{-1} \cdot y \cdot g_2^3(c) \cdot g_2^3(c') = g_2^1(a\cdot a')^{-1} \cdot y \cdot g_2^{3}(c\cdot c'). \end{align*}
Since $(g_3^1, h_1)$ is a map of precrossed modules we get from~\eqref{ax:ch-ga} that for $a \in G_{12}$, $b \in G_{13}$ one has $g_3^1(b^{g_1^2(a)}) = g_3^1(b)^{h_1 g_1^2(a)} = g_3^1(b)^{h_2g_2^1(a)}$.
Since $h_3$ satisfies Peiffer's identity~\eqref{eq:peiffer}, for every $c \in G_{23}$ and $z \in G_3$ one has $z ^{h_2 g_2^3(c)} = z^{ h_3 g_3^2(c)} = g_3^2(c)^{-1} \cdot z \cdot g_3^2(c)$.
Using these identities we obtain that
\begin{multline*}
z'' = g_3^2(c')^{-1} \cdot g_3^1(b')^{-h_2(y')} \cdot z' = \\
= g_3^2(c')^{-1} \cdot g_3^1(b')^{- h_2 \left( g_2^1(a)^{-1} \cdot y g_2^3(c) \right)} g_3^2(c)^{-1} \cdot g_3^1(b)^{-h_2(y)} \cdot z = \\
= g_3^2(c \cdot c')^{-1} \cdot g_3^1(b')^{- h_2 \left( g_2^1(a)^{-1} \cdot y \right)} \cdot g_3^1(b)^{-h_2(y)} \cdot z = \\
= g_3^2(c \cdot c')^{-1} \cdot g_3^1(b \cdot {b'} ^ {g_1^2(a)^{-1}})^{- h_2 \left( y \right)} \cdot z. \qedhere
\end{multline*}
\end{proof}
Since $h$ is constant on the orbits of $\star$-action, $h$ gives rise to a well-defined map $\overline{h} \colon \overline{V} \to G$,
where $\overline{V}$ denotes the set of equivalence classes $V/\sim_W.$
We use the notation $[a, b, c]$ to denote the equivalence class of the triple $(a, b, c) \in V$.

\begin{lemma}\label{lem:one-one-z} Assume that the back face $(f_{12}, f_{23}, g_2^1, g_2^3)$ of~\eqref{eq:cube} is a pullback square.
Let $z \in G_3$ be such that $[1, 1, 1] = [1, 1, z]$.
Then $z \in g_3^1(\Ker(g_1^3)).$ \end{lemma}
\begin{proof} By the definition of congruence relation there exists $(a, b, c)\in W$ such that
\[ (a, b, c) \star (1, 1, 1) = ( g_1^3(b) \cdot g_1^2(a),\ g_2^1(a)^{-1} \cdot g_2^3(c),\ g_3^2(c)^{-1} \cdot g_3^1(b)^{-1}) = (1,1,z). \]
By lemma's assumption there exists $e \in G_{123}$ such that $f_{12}(e) = a$, $f_{23}(e) = c$, hence
\[ 1 = g_1^3(b) \cdot g_1^2(f_{12}(e)) = g_1^3(b \cdot f_{13}(e)),\ z = g_3^2(f_{23}(e))^{-1} \cdot g_3^1(b)^{-1} = g_3^1(b \cdot f_{13}(e))^{-1}. \qedhere\] \end{proof}

\subsection{First reductions} \label{subsec:structure-theorem-overview}
%sSignificant progress towards proving~\cref{thm:horrocks-k2} has already been achieved in~\cite{LS20}.
%In this subsection we reduce~\cref{thm:horrocks-k2} to a specific technical statement which will be addressed in the following sections.

The following lemma, which is based on the argument of~\cite[Theorem~2]{LS20},
provides the first key reduction in the proof of~\cref{thm:horrocks-k2}.
\begin{lemma} \label{lem:first-reduction}
Let $A$ be a local ring with maximal ideal $M$ and residue field $\kappa$.
Let $\Phi$ be as in the statement of~\cref{thm:horrocks-k2}.
Suppose that the canonical homomorphism
\begin{equation} \label{eq:c-surj} C(\Phi, A[X], M[X]) \to C(\Phi, A[X, X\inv], M[X, X\inv]) \end{equation}
is surjective.
Then the square~\eqref{eq:P1-square} is pullback and Horrocks theorem for $\K_2$ holds.
\end{lemma}
\begin{proof}
Denote by $R$ the Laurent polynomial ring $A[X, X\inv]$ and by $B$ the subring $A[X\inv] + M[X] \subseteq R$.
We also set $I \subseteq M[X, X\inv]$, which is clearly an ideal of both $B$ and $R$.
Consider the following diagram with rows obtained from~\eqref{eq:relative-Steinberg}:
\[\begin{tikzcd}
C(\Phi, B, I) \arrow{r} \arrow[d, swap, twoheadrightarrow] & \St(\Phi, B, I) \arrow{r}{\mu_B} \arrow{d} & \St(\Phi, B) \arrow{r} \arrow{d} & \St(\Phi, \kappa[X^{-1}]) \arrow[hookrightarrow]{d} \\
C(\Phi, R, I) \arrow{r} & \St(\Phi, R, I) \arrow{r}{\mu_R} \arrow[ur, "t", dashrightarrow] & \St(\Phi, R) \arrow{r} & \St(\Phi, \kappa[X, X^{-1}]).
\end{tikzcd}\]
The construction of the lifting $t$ is discussed in~\cite[\S~2]{LS17}, cf. also~\cite[Lemma~3.3]{LS20}.
The right-hand side vertical arrow is injective by~\cite[Lemma~2.2]{LS20}.
The left-hand side vertical arrow is surjective since the composite arrow
\[\xymatrix{ C(\Phi, A[X], M[X]) \ar[r] & C(\Phi, B, I) \ar[r] & C(\Phi, R, I) }\]
is surjective by lemma's assumption.
It remains to repeat the diagram chasing argument of~\cite[Theorem~1]{LS20} to conclude that the homomorphism $\St(\Phi, B) \to \St(\Phi, A[X, X\inv])$ is injective.
The assertion of the lemma then follows from~\cite[Theorem~3]{LS20}.
\end{proof}
The proof of Horrocks theorem is, thus, reduced to showing that~\eqref{eq:c-surj} is surjective for every local domain $(A, M)$.

Our next objective is to present a second key reduction in the proof of Horrocks' theorem.
Before proceeding further, we introduce concise notation for certain Steinberg groups and their subgroups.
Specifically, let
{\allowdisplaybreaks\begin{align*}
G     =& \St(\Phi, A[X, X^{-1}]),\\
G^+   =& \St(\Phi, A[X]),\\
B     =& \UU(\Phi^+, A[X, X\inv]) \rtimes \StH(\Phi, A[X, X\inv]) \leq G,\\
G_M   =& \St(\Phi, A[X, X^{-1}], M[X, X^{-1}]),\\
U^+   =& \UU(\Phi^+, A[X]),\\
G^+_M =& \St(\Phi, A[X], M[X]),\\
B_M   =& \UU(\Phi^+, M[X, X^{-1}]) \times \{X, 1+M\} \leq G_M,\\
U^+_M =& \UU(\Phi^+, M[X]),
\end{align*}}
where $\{X, 1+M\}$ denotes the image of the homomorphism $\{X, -\}_{r}$ from~\eqref{eq:relative-symbol}.

The above groups can be organized into the following commutative diagram:
\begin{equation} \label{eq:cube-Steinberg} \xymatrix{
U^+_M \ar@{^{(}->}[rr] \ar@{^{(}->}[dd] \ar@{^{(}->}[rd] &                        & B_M \ar[dd]^(.3){g_B^M} \ar@{^{(}->}[rd]^{g^B_M} &           \\
& G^+_M \ar[rr]^(.3){g^+_M} \ar^(.3){g^M_+}[dd] &                   & G_M \ar[dd]_{h_M} \\
U^+ \ar@{^{(}->}[rr]^(.3){g_B^+} \ar@{^{(}->}[rd]_{g_+^B}          &                        & B \ar@{^{(}->}[rd]_{h_B}       &           \\
& G^+ \ar[rr]_{h_+}              &                   & G.}\end{equation}
Here $g^M_+$ and $h_M$ are simply renamed homomorphisms $\mu$ from~\eqref{eq:relative-Steinberg}.
The homomorphism $g^M_B$ is also induced by $\mu$.
Homomorphisms $g_M^+$ and $h_+$ are induced by the ring embedding $A[X] \to A[X, X\inv]$.
The other homomorphisms on the diagram are all obvious subgroup embeddings.

Set $V = G^+\times B \times G_M,$ $W = U^+\times G^+_M \times B_M.$
Both homomorphisms $g_+^M$ and $h_M$ are crossed modules by~\cref{lem:rel-Steinberg-crossed-module}.
Notice that the actions of $G$ on $G_M$ and the action of $G^+$ on $G_M^+$ are both right conjugation actions.
From this and the functoriality of the relative Steinberg group construction (cf. e.\,g.~\cite[\S~3]{S15}) it follows that $(g^+_M, h_+)$ is a map of precrossed modules.
Thus, we find ourselves in the situation of~\cref{subsec:triples}.

\begin{lemma} \label{lem:second-reduction}
Suppose that $A$ is a domain.
Suppose that $\overline{V} = V/\sim_W$ admits an action of $G$ such that
for $g \in G_M$ one has \[ h_M(g) \cdot [1, 1, 1] = [1, 1, g]. \]
Then the canonical homomorphism~\eqref{eq:c-surj} is surjective.
\end{lemma}
\begin{proof}
Since $A$ is a domain, $g^M_B$ is injective by~\cref{lem:symbols}.
Thus, the back face of~\eqref{eq:cube-Steinberg} is pull-back.
Notice that for any $g \in C(\Phi, A[X, X\inv], M[X, X\inv]) \subseteq G_M$ one has
\[ [1, 1, 1] = h_M(g) \cdot [1, 1, 1] = [1, 1, g].\]
It is clear now that $g$ lies in the image of $C(\Phi, A[X], M[X])$ under $g^+_M$ by~\cref{lem:one-one-z}.
\end{proof}

\begin{rem}
Notice that in the linear case the symbol homomorphism \[A^\times \to \K_2(A[X, X\inv]),\ a \mapsto \{a, X\}\] admits section
for a general commutative ring $A$ by~\cite{Wa71}.
Thus, in the linear case the assertion of the lemma can be proven by the same argument without the assumption that $A$ is a domain, cf.~\cite[Lemma~3.1g]{Tu83}.
\end{rem}

\subsection{Construction of an $\varpi_k$-pair} \label{subsec:construction-sigma}
Throughout this section $\Phi$ denotes a simply laced root system of rank $\ell \geq 3$ and not of type $\rE_8$ and $k$ is the index of
a simple root $\alpha_k$ such that the corresponding weight $\varpi_k$ is a microweight.
The aim of this section is to prove the existence of a $\varpi_k$-pair for $\Phi$ in the sense of~\cref{dfn:delta-pair}
under the additional assumption that $A$ is a local ring.

First of all, recall that the ideal $XA[X]$ is a splitting ideal of $A[X]$, therefore
$\St(\Phi, A[X])$ decomposes as $\St(\Phi, A[X], XA[X]) \rtimes \St(\Phi, A)$.
Since $\varpi_k$ is a microweight subgroups $N_{\pm \varpi_k}$ both contain $\St(\Phi, A[X], XA[X])$.
In fact, $N_{\varpi_k}$ (resp. $N_{-\varpi_k}$) decomposes into the semidirect product of $N_0 := \St(\Phi, A[X], XA[X])$ and the parabolic subgroup $P_k^+ \leq \St(\Phi, A)$ (resp. $P_k^-\leq \St(\Phi, A)$).
%Denote by $\iota$ the natural embedding $N_0 \hookrightarrow N_\omega$.
Recall that $P_k^\pm$ is generated by $x_\alpha(a)$, $a \in A$ for all $\alpha \in \Delta_k \sqcup \Sigma^\pm_k$.

\begin{lemma} \label{lem:relative-reduction}
For $\omega = \pm \varpi_k$ the homomorphism $\sigma(\omega)\colon N_\omega \to N_{-\omega}$ exists if and only if its restriction to $N_0$ exists.
\end{lemma}
\begin{proof}
The necessity of this condition is obvious.
To demonstrate sufficiency, we focus on the case $\omega = \varpi_k$, as the argument for $\omega = -\varpi_k$ is analogous.

By the universal property of semidirect products, for any group $H$ acting on a group $N$, and for any homomorphisms $f_N\colon N \to G$ and $f_H\colon H \to G$ satisfying the coherence condition
\begin{equation}
\label{eq:coherence-condition}
f_N({}^hn) = {}^{f_H(h)} f_N(n), \quad n \in N, \; h \in H,
\end{equation}
there exists a unique homomorphism $f\colon N \rtimes H \to G$ extending both $f_N$ and $f_H$.

Using this universal property, the construction of \( \sigma(\varpi_k) \colon N_{\varpi_k} \to N_{-\varpi_k} \) reduces to the construction of
restrictions \( \sigma(\varpi_k)_{P_k^+} \), \( \sigma(\varpi_k)_{N_0} \), and the verification of~\eqref{eq:coherence-condition}.

We only need to construct \( \sigma(\varpi_k)_{P_k^+} \).
From the Levi decomposition~\eqref{eq:levi-decomp}, we know that \( P_k^+ = L_k \ltimes U_k^+ \), where \( L_k = \mathrm{Im}(\St(\Delta_k, A) \to \St(\Phi, A)) \) and
\( U_k^+ = \UU(\Sigma_k^+, A) \).
Let
\begin{equation} \label{eq:sigma-Pk}
\sigma(\varpi_k)_{L_k} \coloneqq \mathrm{id}_{L_k}, \quad
\sigma(\varpi_k)_{U_k} \left(\prod_{\alpha \in \Sigma_k^+} x_\alpha(a_\alpha)\right) \coloneqq \prod_{\alpha \in \Sigma_k^+} x_\alpha(Xa_\alpha).
\end{equation}
These homomorphisms obviously satisfy~\eqref{eq:coherence-condition}, so the existence of \( \sigma(\varpi_k)_{P_k^+} \) follows from the above universal property.

It remains to verify~\eqref{eq:coherence-condition} for \( \sigma(\varpi_k)_{P_k^+} \) and \( \sigma(\varpi_k)_{N_0} \).
From~\eqref{eq:sigmadef}, we deduce that for any \( \alpha \in \Phi \), \( \beta \in \Sigma_k^- \), and \( f, g \in A[X] \), one has
\begin{equation} \label{eq:charact}
\sigma(\varpi_k)_{N_0}(x_\alpha(Xg)) = x_\alpha(X^{1 + \langle \omega, \alpha \rangle}f), \qquad
\sigma(\varpi_k)_{N_0}(z_\beta(Xg, f)) = z_\beta(g, Xf).
\end{equation}

By~\cref{lem:relative-generators}, the elements $x_\alpha(Xg)$ and $z_\beta(Xg, f)$ form a generating set for \( N_0 \).
It suffices to verify~\eqref{eq:coherence-condition} when \( n \) and \( h \) are generators of their respective groups,
i.e., when \( n \) is one of the above generators of \( N_0 \) and \( h = x_\gamma(a) \) for \( \gamma \in \Delta_k \sqcup \Sigma^+_k \) is a generator of \( P_k^+ \).
This verification is a straightforward calculation using~\eqref{eq:charact}, the commutator relations~\eqref{R2}--\eqref{R3}, and the first two relations from~\cite[Lemma~9]{S15}.
\end{proof}

Our next goal is to prove the existence of an $\varpi_1$-pair for the root system $\Phi$ of type $\rA_\ell$.
The following proposition is based on the construction described in the beginning of~\cite[\S~3]{Tu83}.
Its key novelty is that it also addresses the case \( \ell = 3 \) omitted in~\cite{Tu83}, which will be crucial for the subsequent construction of \( \varpi_k \)-pairs for other types of \( \Phi \) later in this subsection.
For the remainder of this subsection we assume that the base ring $A$ is local.

\begin{prop} \label{prop:sigma-construction}
For a local ring $A$ and $\Phi = \rA_\ell$, $\ell + 1 = n \geq 4$ there exists an $\varpi_1$-pair.
\end{prop}

By~\eqref{lem:relative-reduction}, it remains to construct $\sigma(\pm\varpi_1)_{N_0}$ where $N_0 = \St(n, A[X], XA[X])$.
Let \( d_1 \) denote the matrix $d_1(X) = \mathrm{diag}(X, 1, \ldots, 1) \in \GL(n, A[X, X\inv])$.

Define $R = A[X],$ $S = A[X, X\inv]$, $I = XA[X].$
For $u \in \E(n, R) \cdot e_1$ and $v \in I^n$ we can define homomorphisms $\sigma(\pm\varpi_1)_{N_0}$ on the generators of $N_{0}$ using the elements defined in~\cref{dfn:xy-def}:
\begin{align}
\sigma(\varpi_1)_{N_0} \left(F(u, v)\right) \coloneqq X^{d_1^{-1}}(u, v), & \quad \sigma(\varpi_1)_{N_0} \left(S(v, u)\right) \coloneqq Y^{d_1^{-1} \cdot X}(v, u), \label{eq:def-sigma-1} \\
\sigma(-\varpi_1)_{N_0} \left(F(u, v)\right) \coloneqq X^{d_1 \cdot X^{-1}}(u, v),& \quad \sigma(-\varpi_1)_{N_0} \left(S(v, u)\right) \coloneqq Y^{d_1}(v, u). \label{eq:def-sigma-2}
\end{align}
It is easy to see that the above definitions are correct, i.\,e. all conditions on $u$, $v$ and $d_1$ formulated in~\cref{dfn:xy-def} are satisfied.

Before we can proceed with the proof of~\cref{prop:sigma-construction} we need several lemmas.
\begin{lemma}\label{lem:sigma-N0-image}
The image of $\sigma(\varpi_1)_{N_0}$ is contained in $U_1^- \ltimes N_0 \leq N_{-\varpi_1}$.
\end{lemma}
\begin{proof}
Indeed, it suffices to show that $\mathrm{ev}_{X=0}^*(X^{d_1^{-1}}(u, v))$ and $\mathrm{ev}_{X=0}^*(Y^{d^{-1}_1 \cdot X}(v, u))$ belong to $U_1^-$.
After expanding definitions, every factor $x(v, w)$ appearing in the right-hand sides of either~\eqref{eq:X-def} or~\eqref{eq:Y-def}
has the property that $X$ simultaneously divides both the coordinate $v_1$ and the coordinates $w_2, \ldots, w_n$.
Consequently $\mathrm{ev}_{X=0}^* \left(x(v, w)\right) \in U_1^-$, proving the assertion.
\end{proof}

Denote by $H_{1}$ the subgroup of the group of diagonal matrices $H(\Phi, A)$ generated by all $\pi(h_{\alpha_1}(u))$, $u \in A^\times$
(recall that $\pi$ denotes the homomorphism from~\eqref{eq:K1-K2-sequence}).
Also, we denote by $\mathrm{EP}_1^+$ the image of $P_1^+$ under $\pi$.
\begin{lemma} \label{lem:cor-Gauss}
For local $A$ the group $\E(n, A)$ admits the following decomposition:
\[\E(n, A) = \mathrm{EP}_1^+ \cdot H_{1} \cdot U^-_1 \cdot U^+_1. \]
\end{lemma}
\begin{proof}
Recall from~\cite[Theorem~1.1]{Sm12} that for any root system $\Phi$ the elementary subgroup $\E(\Phi, A)$ admits Gauss decomposition:
\[ \E(\Phi, A) = H(\Phi, A) \cdot \UU(\Phi^+, A) \cdot \UU(\Phi^-, A) \cdot \UU(\Phi^+, A). \]
To prove the assertion, it remains to specialize $\Phi = \rA_{n-1}$ and move all root subgroups
$X_\alpha$ with $\alpha \in \Delta_1$ to the first factor using the Levi decomposition.
\end{proof}

\begin{proof}[Proof of~\cref{prop:sigma-construction}]
We need to show that $\sigma(\varpi_1)_{N_0}$ respects the relations~\eqref{add4}--\eqref{coef-move}.
For relations~\eqref{add4}--\eqref{add5} this is an immediate corollary of~\cref{itm:xsmall-additivity} and \cref{lem:xy-wd}.
By~\cref{lem:xy-conj} for $g \in N_{-\varpi_1}$ one has
\begin{equation}
\label{eq:xy-conj-n1}
g \cdot X^{d_1^{-1}}(u', v') \cdot g^{-1} = X^{d_1^{-1}}(mu', m^*v'), \text{ where } m = d_1^{-1} \cdot \pi(g) \cdot d_1.
\end{equation}
To obtain that $\sigma(\varpi_1)_{N_0}$ respects~\eqref{conj3} it remains to specialize the above equality setting $g = X^{d_1^{-1}}(u, v)$.

Let us verify that $\sigma(\varpi_1)_{N_0}$ respects~\eqref{coef-move}, i.\,e. that for $a\in XA[X]$ and $m \in \E(n, A[X])$ one has
$X^{d_1^{-1}}(me_1, m^*e_2 a) = Y^{d_1^{-1} \cdot X}(me_1 a, m^* e_2)$.
First, we verify this in the special case when $m \in \E(n, A[X])$ belongs to the subset
\[G_0 = H_{1} \cdot U^-_1 \cdot U^+_1.\]
It is easy to check that in this case the only nonzero components of $u = m^* e_2$ are $u_1$ and $u_2$.
Decompose $v = m e_1$ into a sum $v' + v''$, where $v' = (v_1, v_2, 0, \ldots 0)^t,$ $v'' = (0, 0, v_3, \ldots v_n)^t$.
Since $v^t u = 0$ we obtain that $v', v'' \in D(u)$.
Now from Lemmas~\ref{lem:xy-wd}--\ref{lem:xy-conj} we obtain for any $a \in XA[X]$ that
\begin{multline}
\label{eq:special-case}
X^{d_1^{-1}}(me_1, m^*e_2 a) =
x(d_1 \cdot v, d_{1}^{-1} \cdot u a) =
x(d_1 \cdot X^{-1} va, d_1^{-1} \cdot X u) = \\
= x(d_1 X^{-1} \cdot v'a, d_1^{-1}\cdot X u) \cdot x(d_1 X^{-1}\cdot v''a, d_1^{-1} \cdot X u) =
Y^{d_1^{-1} \cdot X}(me_{1}a, m^* e_2)
\end{multline}

To obtain the assertion in the general case notice that by~\cref{lem:cor-Gauss} one has $\E(n, A[X]) = \pi(N_{\varpi_1}) \cdot G_0$.
Now factor $m \in \E(n, A[X])$ as $\pi(n) \cdot h$ for some $n\in N_{\varpi_1}$ and $h \in G_0$.
Since $\pi(n) = d_1^{-1} \cdot \pi(g) \cdot d_1$ for some $g \in N_{-\varpi_1}$ it remains to apply~\eqref{eq:xy-conj-n1}, \eqref{eq:special-case} and~\cref{lem:xy-conj}:
\begin{multline}
\nonumber X^{d_1^{-1}}(me_1, m^*e_{2}a) = {}^{g}(X^{d_1^{-1}}(he_1, h^*e_{2}a)) = \\
= {}^{g}(Y^{d_1^{-1} \cdot X}(he_{1}a, h^*e_2)) = Y^{d_1^{-1} \cdot X}(me_{1}a, m^{*} e_{2}).
\end{multline}
The construction of $\sigma(\varpi_1)$ is now complete.
The argument for $\sigma(-\varpi_1)$ is analogous.
\end{proof}

\begin{cor} \label{cor:a3-microweight}
For a local ring $A$ and $\Phi = \rA_3$ an $\omega$-pair exists for any weight $\omega$ lying the $W(\Phi)$-orbit of a microweight.
\end{cor}
\begin{proof}
By~\cref{lem:delta-weyl} that it suffices to prove the existence of an $\omega$-pair in the case when $\omega$ is a fundamental weight, i.\,e. $\omega = \varpi_{i}$ for some $1 \leq i \leq 3$.

In \cref{prop:sigma-construction}, we have established the existence of an $\omega$-pair for $\omega = \varpi_1$
(and by symmetry, also for $\omega = \varpi_3$, cf.~\cref{exm:chi-linear}).
Thus, it remains to do the same for $\omega = \varpi_{2}$.

By~\cref{lem:relative-reduction} we only need to construct the restriction $\sigma(\varpi_2)_{N_0}$, where $N_0 = \St(4, A[X], XA[X])$.
Specifically, we define $\sigma(\varpi_2)_{N_0}$ as the composition
$\sigma(\varpi_1) \cdot \sigma((12) \cdot \varpi_1)_{N_0}$.
By \cref{lem:sigma-N0-image} and~\eqref{eq:sigma-gen}, the image of $\sigma((12) \cdot \varpi_1)_{N_0}$ is contained in
${}^{w_{12}(1)}U_1^- \ltimes N_0 \leq N_{\varpi_1},$ so the above homomorphisms can indeed be composed correctly.
Using~\eqref{eq:sigmadef} and the identity
$\varpi_2 = (12) \cdot \varpi_1 + \varpi_1$ it is straightforward to show that the image of $\sigma(\varpi_2)_{N_0}$ is contained in $N_{-\varpi_2}$.
\end{proof}

Now we can prove the main result of this subsection.
\begin{thm} \label{thm:pairconstr}
Let $\Phi$ be a root system of type $\rD_{\geq 4}$, $\rE_6$ or $\rE_7$ and $A$ be a local ring.
Denote by $k$ the index of the simple root $\alpha_k$, whose weight $\varpi_k$ is a microweight
($k$ is the index of the vertex marked green on the Dynkin diagram from Figure 1 in~\cref{subsec:curtis-tits}).
Then there exists a $\varpi_k$-pair.
\end{thm}
\begin{proof}
By~\cref{lem:relative-reduction} it suffices to constuct $\sigma({\varpi_k})_{N_0}$.
In our situation \cref{thm:relPres} asserts that $N_0 = \St(\Phi, A[X], XA[X])$ can be decomposed into the free product of subgroups $\St(\Psi, A[X], XA[X])$,
where $\Psi$ ranges over the set $A_3(\Phi)$ of root subsystems of type $\rA_3$ in $\Phi$, amalgamated over embeddings of generators $z_\alpha(Xf, g)$ into different groups $\St(\Psi, A[X], XA[X])$, whose root system $\Psi \in A_3(\Phi)$ contains $\alpha$.
Thus, by the universal property of coproducts it remains to construct $\sigma(\varpi_k)_\Psi \colon \St(\Psi, A[X], XA[X]) \to N_{-\varpi_k}$ and then
verify that these homomorphisms agree on the generators $z_\alpha(Xf, g)$ of different $\St(\Psi, A[X], XA[X])$.

In the case $\Psi \subseteq \Delta_k$ the homomorphism $\sigma(\varpi_k)_\Psi$ should act as identity, so we define it as the canonical map $\St(\Psi, A[X], XA[X]) \to N_{-\varpi_k}$
induced by the embedding $\Psi \subseteq \Phi$.

Now consider the nontrivial case $\Psi \not\subseteq \Delta_k$.
It is clear that there is a unique nonzero weight $\omega \in P(\Psi)$ determined by
\begin{equation} \label{eq:Psi-choice} (\omega, \alpha) = (\varpi_k, \alpha),\ \alpha \in \Psi. \end{equation}
By definition, the weight $\omega$ is a microweight for $\Psi$.
By~\cref{cor:a3-microweight} there exists an $\omega$-pair $\xymatrix{ \sigma(\omega)\colon N_{\omega, \Psi} \ar[r] & \ar@<-1.0ex>[l] N_{-\omega, \Psi}\colon \sigma(-\omega) }$.

Denote by $\iota$ the homomorphism $N_{-\omega, \Psi} \to N_{-\varpi_k, \Phi}$ induced by the embedding of Steinberg groups $\St(\Psi, A[X]) \to \St(\Phi, A[X])$ (it is well-defined by the definition of $N_\omega$ and~\eqref{eq:Psi-choice}).
The required homomorphisms $\sigma(\varpi_k)_\Psi$ now can be defined as the composition $\iota \sigma(\omega)$.
The fact that $\sigma(\varpi_k)_\Psi(z_\alpha(Xf, g))$ does not depend on $\Psi$ is obvious from the construction of $\sigma(\varpi_k)_\Psi$.
\end{proof}

\subsection{Construction of the action of $\St(\Phi, A[X, X\inv])$ on $V$} \label{sec:construction-delta}
From now on we assume that $A$ is a local domain.
Recall from \cref{subsec:structure-theorem-overview} that \[V = G^+ \times B \times G_M,\ W = U^+ \times G_M^+ \times B_M\]
and that $W$ acts upon $V$ via $\star$-action defined in~\cref{subsec:triples}.

For a weight $\omega \in P(\Phi)$ denote by $V_\omega$ the subset of $V$ consisting of those triples $v = (x, y, z)$ for which $x \in N_\omega\cdot \StW(\Phi, A)$.
Throughout this section we assume that $\Phi$ is a root system of type $\rD_{\geq 4}$, $\rE_6$ or $\rE_7$.

\begin{dfn} \label{sigma-def}
Let $k$ be the index of the vertex marked green on the Dynkin diagram from Figure 1 in~\cref{subsec:curtis-tits} and let $\omega = \pm \varpi_k$,
We define the function $\delta_\omega \colon V_\omega \to V_{-\omega}$ via the following formula:
\begin{equation} \label{eq:sigma-def} (x_1 \cdot x_2, y, z) \mapsto (\sigma(\omega)(x_1)\cdot x_2, x_2^{-1} \cdot \chi_{\omega, X}(x_2 \cdot y), \widetilde{\chi}_{\omega, X}(z)), \end{equation}
where $x_1 \in N_\omega$, $x_2 \in \StW(\Phi, A)$, $y \in B$, $z \in G_M$, and
$\chi_{\omega, X}$, $\widetilde{\chi}_{\omega, X}$ denote automorphisms defined in~\eqref{eq:chi-def} and~\cref{lem:relative-chi}, respectively,
while $\sigma(\omega)$ denotes the homomorphism defined in~\cref{thm:pairconstr}.
\end{dfn}
By~\cref{lem:winv-chiw} the factor $x_2^{-1} \cdot \chi_{\omega, X}(x_2)$ lies in $\StH(\Phi, A[X^{\pm 1}])$,
so the second component of the above triple belongs to $B$.
At this point, it is not immediately clear why the above definition does not depend on the choice of decomposition $x = x_1 \cdot x_2$.

Denote by $\Lambda$ the orbit of $\varpi_k$ under $W(\Phi)$.
Since $W(\Delta_k)$ stabilizes $\varpi_k$ it follows that the coset set $W(\Phi)/W(\Delta_k)$ is in one-to-one correspondence with $\Lambda$.

We denote by $V$ the microweight representation of $\Gsc(\Phi, A)$ with the highest weight $\varpi_k$ (see e.\,g. \cite[\S~2]{Ge17} or \cite[\S~1.1]{V00}).
This representation has dimension $|\Lambda|$ and we denote by $v^\lambda$ its basis vector corresponding to $\lambda \in \Lambda$.
We also denote by $v^+$ the highest weight vector of this representation, i.\,e. the basis vector $v^{\varpi_k}$.
A choice of the system of positive roots specifies a strict partial order on $\Lambda$.
By definition, $\lambda > \mu$ if $\lambda - \mu$ is a sum of positive roots.

Recall that for every $w \in W(\Phi, A)$ there exists a unique $\lambda \in \Lambda$ and $u \in A^\times $ such that
\begin{equation}\label{eq:w-highestweight} w \cdot v^\lambda = u^{-1} v^+. \end{equation}
Conversely, for every pair $(u, \lambda) \in A^\times \times \Lambda$ there exists some element $w = w_{\lambda, u} \in W(\Phi, A)$ such that~\eqref{eq:w-highestweight} holds.
If $\lambda = \varpi_k$ the element $w_{\lambda, u}$ can be chosen as $h_\alpha(u)$ for some $\alpha \in \Sigma_k^-$, $u \in A^\times$, cf.~\cite[Lemma~7]{V00}.
If $\lambda \neq \varpi_k$ by~\cite[Lemma~6]{V00} the element $w_{\lambda, u}$ can be chosen as a product of \textit{at most} $n$ elements
$w_{\beta_i}(v_i)$, $1 \leq i \leq n$, for $\beta_i\in\Sigma_k^-$ such that $\beta_i$ are mutually orthogonal and $\prod_i v_i = \pm u$ and $\varpi_k + \sum_i \beta_i = \lambda$ where
\begin{itemize}
\item $n=2$ if $\Phi = \rD_\ell, \rE_6$,
\item $n=3$ if $\Phi = \rE_7$.
\end{itemize}
Formula~\eqref{eq:w-highestweight} follows from the description of the action of $w_\beta(v)$ on $V$ given by~\cite[Lemma~6]{V00}.
It is also clear that $\overline{w_{\lambda, u}} \in W(\Phi)$ swaps weights $\varpi_k$ and $\lambda$.

Before we proceed further recall that the canonical homomorphism
$\K_2(\Delta_k, A) \to \K_2(\Phi, A)$ induced by the embedding $\Delta_k \subseteq \Phi$ is surjective by~\cite[Theorem~2.13]{Ste73}.

Notice also that
\begin{equation} \label{eq:WP} W(\Phi, A) \cap \pi(P_k^+) \subseteq W(\Delta_k, A). \end{equation}

Denote by $\overline{W}(\Delta_k, A)$ the image of $\StW(\Delta_k, A)$ in $\St(\Phi, A)$ (or, equivalently, the subgroup generated by $w_\alpha(u)$, $u\in A^\times,$ $\alpha \in \Delta_k$).
\begin{lemma} \label{lem:can-repr1}
We have $\StW(\Phi, A) \cap N_{\varpi_k} = \overline{W}(\Delta_k, A).$
\end{lemma}
\begin{proof}
Clearly, $\overline{W}(\Delta_k, A) \subseteq \StW(\Phi, A) \cap N_{\varpi_k}$.
Let us now establish the reverse inclusion.

Consider an element $g \in \StW(\Phi, A) \cap N_{\varpi_k}$.
By~\eqref{eq:WP} the element $\pi(g)$ belongs to the subgroup $W(\Delta_k, A)$.
Consequently, $g \in \K_2(\Phi, A) \cdot \overline{W}(\Delta_k, A)$.
It remains to note that $\K_2(\Phi, A) \subseteq \overline{W}(\Delta_k, A)$ since the map $\K_2(\Delta_k, A) \to \K_2(\Phi, A)$ is surjective.
\end{proof}

\begin{lemma} \label{lem:can-repr2}
The map $(u, \lambda) \mapsto w_{\lambda, u}$ establishes an isomorphism of the set $A^\times \times \Lambda$ and the coset set $\StW(\Phi, A) / \overline{W}(\Delta_k, A)$.
\end{lemma}
\begin{proof}
Since $\K_2(\Delta_k, A) \to \K_2(\Phi, A)$ is surjective, the coset set $\StW(\Phi, A) / \overline{W}(\Delta_k, A)$ is in bijective correspondence with the coset set
$W(\Phi, A)/W(\Delta_k, A)$.
We claim that $W(\Phi, A)/W(\Delta_k, A)$ is isomorphic to $A^\times \times \Lambda$
via the mapping $w W(\Delta_k, A) \mapsto (u, \lambda)$, where $u$ is determined from~\eqref{eq:w-highestweight}.
This map is obviously surjective, its injectivity follows from Chevalley--Matsumoto decomposition~\cite[Theorem~1.3]{St78} and~\eqref{eq:WP}.
\end{proof}

\begin{cor} \label{cor:can-repr}
For $\omega = \pm \varpi_k$ one has \[N_{\omega} \cdot \StW(\Phi, A) = \bigsqcup\limits_{(u, \lambda) \in A^\times \times \Lambda} N_{\omega} \cdot w_{\lambda, u}. \]
\end{cor}
\begin{proof}
Follows from the coset decomposition for $\StW(\Phi, A)$, and Lemmas~\ref{lem:can-repr1}--\ref{lem:can-repr2}.
\end{proof}

\begin{cor} \label{cor:can-repr2}
For $\omega = \pm\varpi_k$ the right-hand side of~\eqref{eq:sigma-def} does not depend on the choice of decomposition $x = x_1 \cdot x_2$, so $\delta_\omega$ is defined unambiguously.
Moreover, $\delta_{\varpi_k}$ and $\delta_{-\varpi_k}$ are mutually inverse.
\end{cor}
\begin{proof}
The first assertion follows from~\cref{lem:can-repr1} and the observation that, by their definitions, $\sigma(\omega)$ and $\chi_{\omega, X}$ act identically on $\overline{W}(\Delta_k, A)$.
The second assertion can be proved by a straightforward calculation using~\eqref{eq:sigma-def} and the definition of an $\omega$-pair.
\end{proof}

\begin{lemma} \label{lem:v-correctness1}
For every $v \in V$ and $\omega = \pm\varpi_k$ there exists $(a, b, 1) \in W$ such that $(a, b, 1) \star v \in V_\omega$.
\end{lemma}
\begin{proof}
Let $P$ be $\Pi$ or $-\Pi$ depending on whether $\omega$ is $\varpi_k$ or $-\varpi_k$.
Thus, $\Phi^+_P$ is either $\Phi^+$ or $\Phi^-$.
From~\cref{lem:bruhat} and the normality of subgroups $\St(\Phi, A[X], XA[X]) \leq G^+$, $\overline{\St}(\Phi, A, M) \leq \St(\Phi, A)$ we obtain that
\[G^+ = \UU(\Phi^+_P, A)\cdot \St(\Phi, A[X], XA[X]) \cdot \StW(\Phi,A)\cdot \overline{\St}(\Phi, A, M) \cdot\UU(\Phi^+, A).\]
The first two factors in the above decomposition are contained in $N_\omega$, so the assertion follows from the definition of $\star$.
\end{proof}

For $\lambda \in \Lambda$ and a commutative ring $R$ we denote by $\Stab_{\lambda}(R)$ the subgroup of $\Gsc(\Phi, R)$
consisting of elements stabilizing the basis vector $v^\lambda$ of $V$.
%Recall that the projection $\pi \colon \St(\Phi, R) \to \G_{sc}(\Phi, R)$ is always injective on unipotent subgroups.
Moreover, it is clear that $\mathrm{ev}_{X=0}^*(N_\lambda) = \pi^{-1}(\Stab_\lambda(A))$ and that
$(\Stab_{\varpi_k}(A))^{w_{\lambda, u}} = \Stab_\lambda(A)$.
\begin{lemma} \label{lem:q}
Suppose that $a \in \UU(\Phi^+, A)$ is such that its image $\pi(\overline{a})$ belongs to $\Stab_{\lambda}(A/M)$ for some $\lambda \in \Lambda$.
Then there exists $q \in \UU(\Phi^+, M)$ such that $\pi(aq)$ belongs to $\Stab_{\lambda}(A)$.
\end{lemma}
\begin{proof}
Set $b = {}^{w_{\lambda, u}} a$.
Notice that $b \cdot v_+ = v^+ + \sum_{\mu < \varpi_k} c_\mu v^\mu = 1$, therefore by the Chevalley--Matsumoto decomposition~\cite[Theorem~1.3]{St78} we can write
$b = b_1 \cdot b_2$ for some $b_1$ such that $\pi(b_1) \in \Stab_{\varpi_k}(A)$ and $b_2 \in \UU(\Sigma_k^-, A)$.
It is also clear from our assumptions that $c_\mu \in M$ hence $b_2 = \prod_{\alpha \in \Sigma_k^-}x_{\alpha}(c_{\varpi_k + \alpha})$, in fact, belongs to $\UU(\Sigma_k^-, M)$.
Set $a_i = b_i^{w_{\lambda, u}}$.
It is clear that $a = a_1 \cdot a_2$  for some $a_1, a_2$ satisfying $\pi(a_1) \in \Stab_\lambda(A)$, $\pi(a_2) \in \UU(\Phi^+ \cap \overline{w_{\lambda, u}^{-1}}(\Sigma_k^-), M)$,
from which the assertion of the lemma follows.
\end{proof}

\begin{prop} \label{prop:v-correctness2}
For $\omega = \pm \varpi_k$ and every $w \in W$, $v \in V_\omega$ such that $v' = w \star v \in V_\omega$ there exists $w' \in W$ such that
$\delta_\omega(w \star v) = w' \star \delta_\omega(v)$.
\end{prop}
\begin{proof}
It suffices to consider the case $\omega = \varpi_k$.
Suppose that $v = (x_1 x_2, y, z) = (a, b, 1) \star v'$, where $v' = (x_1' x_2', y', z')$.
We have that
\begin{equation}\label{eq:x-xprime} x_1 x_2 = x_1' x_2' \cdot \mu(b) \cdot a  \text{ for some }x_1, x_1' \in N_\omega,\ x_2, x_2' \in \StW(\Phi, A) \end{equation}
and some $a \in \UU(\Phi^+, A[X]),\ b \in \St(\Phi, A[X], M[X])$.
Thanks to~\cref{cor:can-repr} we may assume, without loss of generality, that $x_2 = w_{\lambda, u}$, $x_2' = w_{\lambda', u'}$
for some $u, u' \in A^\times$, $\lambda, \lambda' \in \Lambda$.

Evaluating both sides of~\eqref{eq:x-xprime} at $X=0$ and projecting the equality into $\St(\Phi, \kappa)$ we obtain that
\begin{equation} \label{eq:sides}
\overline{{x_2'}^{-1}} \cdot \overline{\mathrm{ev}_{X=0}^*({x_1'}^{-1}x_1)} \cdot \overline{x_2} = \overline{\mathrm{ev}_{X=0}^*(a)},
\end{equation}
where by $\overline{x}$ we denote the image of $x$ in $\St(\Phi, \kappa)$.

Notice that the middle factor in the left-hand side of~\eqref{eq:sides} stabilizes the highest-weight vector $v^+$,
therefore we can compute the images of the vector $v^{\lambda'}$ under both sides of~\eqref{eq:sides} using~\eqref{eq:w-highestweight} as follows:
\begin{align} \overline{{x_2'}^{-1}} \cdot \overline{\mathrm{ev}_{X=0}^*({x_1'}^{-1}x_1)} \cdot \overline{x_2} \cdot v^{\lambda'} = \overline{{x_2'}^{-1}} \cdot \overline{{u}^{-1}} \cdot v^+ &= \overline{u^{-1}}\overline{u'} \cdot v^\lambda, \\
\overline{\mathrm{ev}_{X=0}^*(a)} \cdot v^{\lambda'} &= v^{\lambda'} + \sum_{\mu > \lambda'} \overline{a_\mu} v^\mu \label{eq:right-side} \end{align}
for some $a_\mu \in A$.
Thus, from~\eqref{eq:sides}--\eqref{eq:right-side} we conclude that $\lambda = \lambda'$ and that $1 - u/u' \in M$ and also that $a_\mu \in M$.
Using~\cite[Lemma~7]{V00} and~\eqref{eq:w-highestweight} we can find $m \in M$ such that the element
$h_{\alpha_k}(1+m) \cdot x_2' \cdot x_2^{-1}$ stabilizes the highest weight vector $v^+$.
By the same argument as in the proof of~\cref{lem:can-repr2} we conclude that it belongs to $\overline{W}(\Delta_k, A)$.
Consequently, we can write $x_2 = c \cdot h_{\alpha_k}(1+m) \cdot x_2'$ for some $c \in \overline{W}(\Delta_k, A)$.

Also we conclude from all the equalities obtained up to this point that
\begin{equation} {}^{x_2'}(\mu(b) \cdot a) \cdot h_{\alpha_k}^{-1}(1+m) = {x_1'}^{-1} \cdot x_1 \cdot c \in N_\omega. \end{equation}

Recalling that $\overline{w_{\lambda, u}^{-1}} \in W(\Phi)$ swaps $\omega$ and $\lambda$
we obtain from~\eqref{eq:sigma-gen} and \cref{lem:delta-weyl} that
\begin{equation} \label{eq:delta-weyl-x2prime} {}^{x_2'}N_\lambda = N_\omega\text{ and } {}^{x_2'}\sigma(\lambda)(g) = \sigma(\omega)({}^{x_2'}g)\text{ for } g\in N_\lambda.\end{equation}

Consequently, we obtain that
$a \cdot \mu(b^a) \cdot h_{\alpha_k}^{-{x_2'}}(1+m) \in {N_{\lambda}}$.

By~\eqref{eq:right-side} and~\cref{lem:q} we can find $q \in \UU(\Phi^+, M)$ such that $\pi(\mathrm{ev}_{X=0}^*(a) \cdot q) \in \Stab_\lambda(A)$ and hence
that $a \cdot q \in N_\lambda$.
Now we can compute the first coordinate of $\delta_\omega \cdot v$ as follows:
\begin{multline} \label{eq:intermediate}
\sigma(\omega)(x_1) \cdot x_2 = \sigma(\omega)(x_1'x_2' \cdot \mu(b) \cdot a \cdot x_2^{-1}) \cdot x_2 = \\
= \sigma(\omega)(x_1'x_2' \cdot \mu(b) \cdot a \cdot h^{-{x_2'}}_{\alpha_k}(1+m) \cdot {x_2'}^{-1} \cdot c^{-1}) \cdot c \cdot x_2' \cdot h_{\alpha_k}^{x_2'}(1+m) = \\
= \sigma(\omega)\left(x_1'\cdot {}^{x_2'} (aq \cdot q^{-1} \mu(b^a) \cdot h^{-x_2'}_\alpha(1+m))\right) \cdot x_2' \cdot h_{\alpha}^{x_2'}(1+m) = \\
= \sigma(\omega)(x_1') \cdot {}^{x_2'}\sigma(\lambda)(aq) \cdot {}^{x_2'}\sigma(\lambda)(q^{-1} \mu(b^a) \cdot h^{-x_2'}_{\alpha_k}(1+m)) \cdot x_2' \cdot h^{x_2'}_{\alpha_k}(1+m) = \\
= \sigma(\omega)(x_1') \cdot x_2' \cdot \sigma(\lambda)(aq) \cdot \sigma(\lambda)\left(q^{-1} \mu(b^a) \cdot h^{-x_2'}_{\alpha_k}(1+m)\right) \cdot h^{x_2'}_{\alpha_k}(1+m).
\end{multline}

Now we let us compute the second coordinate of $\delta_\omega \cdot v$ using~\eqref{eq:chi-h} and the fact that
$\chi_{\omega, X}$ acts identically on $\StW(\Delta_k, A)$:
\begin{multline}\label{eq:intermediate2}
x_2^{-1}\cdot \chi_{\omega, X}(x_2 \cdot y) = x_2^{-1} \cdot \chi_{\omega, X}(x_2 \cdot a^{-1} \cdot y') = \\
= {x_2'}^{-1} \cdot h_{\alpha_k}^{-1}(1 + m) \cdot c^{-1} \cdot \chi_{\omega, X}(c \cdot h_{\alpha_k}(1 + m) \cdot x_2' \cdot a^{-1} \cdot y') = \\
= {x_2'}^{-1} \cdot h_{\alpha_k}^{-1}(1 + m) \cdot \chi_{\omega, X}(h_{\alpha_k}(1 + m) \cdot x_2' \cdot a^{-1} \cdot y') = \\
= \{X, 1+m\} \cdot {x_2'}^{-1} \chi_{\omega, X}(x_2' \cdot a^{-1} \cdot y').
\end{multline}

It is clear that the product of last two factors in the right-hand side of~\eqref{eq:intermediate} is
the image in $G^+$ of some element $b' \in G_M^+$.
Now we can compute the third coordinate of $(1, {b'}^{-1}, 1) \star (\delta_\omega \cdot v)$ using~\eqref{eq:intermediate2}, \eqref{eq:delta-weyl-x2prime} and~\eqref{eq:chi-h} as follows:
\begin{multline*}
{b'}^{x_2^{-1} \cdot \chi_{\omega, X}(x_2 \cdot y) } \cdot \widetilde{\chi}_{\omega, X}(z) =
{b'}^{{x_2'}^{-1} \cdot \chi_{\omega, X}(x_2' \cdot a^{-1} \cdot y') } \cdot \widetilde{\chi}_{\omega, X}(b^{-y'} \cdot z') = \\
= \left(\sigma(\omega)\left( {}^{x_2'}(q^{-1} \mu(b^a)) \cdot h^{-1}_{\alpha_k}(1+m)\right) \cdot h_{\alpha_k}(1+m)\right)^{\chi_{\omega, X}(x_2' \cdot a^{-1} \cdot y')} \cdot \widetilde{\chi}_{\omega, X}(b^{-y'} \cdot z') = \\
= \{1+m, X\} \cdot \widetilde{\chi}_{\omega, X}\left( \left( q^{-1} \cdot b^a \right)^{a^{-1}y'} \cdot b^{-y'} \cdot z' \right) = \\
= \{1+m, X\} \cdot \widetilde{\chi}_{\omega, X}(q^{-a^{-1}y'} \cdot z').
\end{multline*}
Thus, we obtain that
\begin{multline*}
v'' = (1, {b'}^{-1}, \{1 + m, X\}) \star (\delta_\omega \cdot v) = \\ =
(\sigma(\omega)(x_1') \cdot x_2' \cdot \sigma(\lambda)(aq), x_2'^{-1} \cdot \chi_{\omega, X}(x_2' \cdot a^{-1} y'), \widetilde{\chi}_{\omega, X}(q^{-a^{-1}y'} \cdot z')).
\end{multline*}
Notice that $\sigma(\lambda)(aq) \in \UU(\Phi^+, A[X])$, while $\chi_{\omega, X}(q^{-a^{-1}y'})\in \UU(\Phi^+, M[X, X\inv])$.
Thus, from the definition of $\star$-action we obtain that
\begin{multline*}
(\sigma(\lambda)(aq)^{-1}, 1, \chi_{\omega, X}(q^{-a^{-1}y'})) \star v'' = \\
= (\sigma(\omega)(x_1') \cdot x_2', \sigma(\lambda)(aq) x_2'^{-1} \cdot \chi_{\omega, X}(x_2' \cdot a^{-1} y') \cdot \chi_{\omega, X}(q^{-a^{-1}y'}), \widetilde{\chi}_{\omega, X}(z')) = \\
= (\sigma(\omega)(x_1') \cdot x_2', x_2'^{-1} \chi_{\omega, X}(x_2' \cdot q^{a^{-1}} \cdot y' \cdot q^{-a^{-1}y'}), \widetilde{\chi}_{\omega, X}(z')) = \\
= \delta_\omega (x_1' \cdot x_2', y', z'). \qedhere
\end{multline*}
\end{proof}

For $\omega = \pm \varpi_k$ the function $\delta_\omega \colon V_{\omega} \to V_{-\omega}$ gives rise to a well-defined function $\overline{\delta}_\omega \colon \overline{V} \to \overline{V}.$
Indeed, by~\cref{lem:v-correctness1} for every $v \in V$ there exists $w \in W$ such that $w \star v \in V_\omega$, therefore
we can define $\overline{\delta_\omega}([v])$ as the equivalence class of $\delta_{\omega}(w \star v)$.
By~\cref{prop:v-correctness2} this definition is correct, i.\,e. does not depend on the choice of $w$.

\begin{prop}
There is a well-defined action of $\St(\Phi, A[X, X\inv])$ on $\overline{V}$ compatible with the action of $\St(\Phi, A[X])$.
\end{prop}
\begin{proof}
Thanks to~\cref{prop:rel-poly-Laurent} it suffices to construct the action of the group $G$ on $\overline{V} = V/\sim_W$.
Here $G$ denotes the free product of $\langle \sigma \rangle$ and $\St(\Phi, A[X])$ amalgamated over the relations~\eqref{eq:sigma-sigma-plus}--\eqref{eq:sigma-delta}.

We let the generator $\sigma$ of $G$ act on $V/\sim_W$ via $\delta_{\varpi_k}$.
By~\cref{cor:can-repr2} $\sigma^{-1}$ acts on $V/\sim_W$ via $\delta_{-\varpi_k}$.
For $\alpha \in \Phi$ we define the translation action of $g \in \St(\Phi, A[X])$ on $V$ via $g \cdot (x, y, z) \coloneqq (gx, y, z)$.
It is easy to see that this definition is correct and gives rise to a well-defined action of $\St(\Phi, A[X])$ on $\overline{V}$.

We need to check that these two actions satisfy relations~\eqref{eq:sigma-sigma-plus}--\eqref{eq:sigma-delta}.
Let us verify~\eqref{eq:sigma-sigma-plus}.
Indeed, for $v = (x_1 \cdot x_2, y, z) \in V_{-\omega_k}$, where $x_1 \in V_{-\varpi_k}$, $x_2 \in \StW(\Phi, A)$, $y \in B$, $z \in G_M$ and $\alpha \in \Sigma_k^+$, $f \in A[X]$
we obtain from~\eqref{eq:sigmadef} and~\eqref{eq:sigma-def} that
\begin{multline*}
\overline{\delta}_{\varpi_k} \cdot x_\alpha(f) \cdot \overline{\delta}_{-\varpi_k} \cdot [x_1 \cdot x_2, y, z] = \\
= \overline{\delta}_{\varpi_k} \cdot x_\alpha(f) \cdot [\sigma(-\varpi_k)(x_1) \cdot x_2, x_2^{-1} \cdot \chi_{-\varpi_k, X}(x_2 \cdot y), \widetilde{\chi}_{-\varpi_k, X}(z)] = \\
= \left[\sigma(\varpi_k)\left(x_{\alpha}(f) \cdot \sigma(-\varpi_k)(x_1)\right) \cdot x_2, x_2^{-1} \cdot \chi_{\varpi_k, X}\left(x_2 \cdot x_2^{-1} \cdot \chi_{-\varpi_k, X}(x_2 \cdot y)\right), z\right] = \\
= [x_\alpha(Xf) \cdot x_1 \cdot x_2, y, z] = x_\alpha(Xf) \cdot [x_1 \cdot x_2, y, z].
\end{multline*}

Relation~\eqref{eq:sigma-sigma-minus} can be verified by the same calculation with $\varpi_k$ and $\Sigma_k^+$ replaced with $-\varpi_k$ and $\Sigma_k^-$, respectively.
Relation~\eqref{eq:sigma-delta} follows from the fact that $\sigma(\pm \varpi_k)$ fixes elements of the image of the canonical homomorphism $\St(\Delta_k, A[X]) \to \St(\Phi, A[X])$.
\end{proof}

Our next goal is to prove~\cref{thm:horrocks-k2}.
In order to achieve this we need several lemmas.

\begin{lemma} \label{lem:family1}
For $z \in \St(\Phi, M[X, X\inv], A[X, X\inv])$, $n, m \geq 0$, $f \in M[X]$, $g \in A[X]$, $\beta \in \Sigma_k^\pm$ one has
\[ z_\beta(X^{-m}f, X^{-n}g) \cdot [1, 1, z] =  [1, 1, z_\beta(X^{-m}f, X^{-n}g) \cdot z].\]
\end{lemma}
\begin{proof}
We will only verify the assertion in the case $\beta \in \Sigma_k^-$
(the case $\beta \in \Sigma_k^+$ can be obtained by replacing everywhere in the proof $\sigma$ with $\sigma^{-1}$).

Notice that in the group $G$ one has
\begin{equation} z_\beta(X^{-m}f, X^{-n}g) =
\sigma^n \cdot x_{-\beta}(-g) \cdot \sigma^{-n-m} \cdot x_{\beta}(f) \cdot \sigma^{n + m} \cdot x_{-\beta}(g) \cdot \sigma^{-n} \end{equation}
For shortness we write $[x, y, z] \rightsquigarrow_{g} [x', y', z']$ to illustrate the fact that
$[x', y', z'] = g \cdot [x, y, z]$.
Direct calculation using~\eqref{eq:sigmadef} and~\eqref{eq:sigma-def} shows that
\begin{multline*}
[1, 1, z] \rightsquigarrow_{\overline{\delta}_{-\varpi_k}^n}
[1, 1, \widetilde{\chi}^n_{-\varpi_k, X}(z)] \rightsquigarrow_{x_{-\beta}(g)}
[x_{-\beta}(g), 1, \widetilde{\chi}_{-\varpi_k, X}^n(z)] = \\
= [1, x_{-\beta}(g), \widetilde{\chi}^n_{-\varpi_k, X}(z)] \rightsquigarrow_{\overline{\delta}_{\varpi_k}^{n + m}}
[1, x_{-\beta}(X^{-n-m}g), \widetilde{\chi}^m_{\varpi_k, X}(z)] \rightsquigarrow_{ x_{\beta}(f) } \\
\rightsquigarrow_{ x_{\beta}(f) } [x_{\beta}(f), x_{-\beta}(X^{-n-m}g), \widetilde{\chi}^m_{\varpi_k, X}(z)] = \\
= [1, x_{-\beta}(X^{-n-m}g), x_{\beta}(f)^{x_{-\beta}(X^{-n-m}g)} \widetilde{\chi}^m_{\varpi_k, X}(z)] \rightsquigarrow_{\overline{\delta}^{n + m}_{-\varpi_k}} \\
\rightsquigarrow_{\delta_{-\varpi_k}^{n + m}} [1, x_{-\beta}(g), x_{\beta}(X^{-n-m}f)^{x_{-\beta}(g)} \widetilde{\chi}^n_{-\varpi_k, X}(z)] \rightsquigarrow_{ x_{-\beta}(-g) } \\
\rightsquigarrow_{ x_{-\beta}(-g) } [1, 1, x_{\beta}(X^{-n-m}f)^{x_{-\beta}(g)} \widetilde{\chi}^n_{-\varpi_k, X}(z)] \rightsquigarrow_{ \overline{\delta}_{\varpi_k}^n }
[1, 1, x_{\beta}(X^{-m}f)^{x_{-\beta}(X^{-n}g)} \cdot z],
\end{multline*}
from which the assertion of the lemma follows.
\end{proof}

\begin{lemma} \label{lem:family2}
For $z \in \St(\Phi, M[X, X\inv], A[X, X\inv])$, $\alpha \in \Delta_k$, $n \geq 0$ and $f \in M[X]$ one has
\begin{equation*} x_\alpha(X^{-n} f) \cdot [1, 1, z] = [1, 1, x_\alpha(X^{-n} f) \cdot z]. \end{equation*}
\end{lemma}
\begin{proof}
We can find $\beta \in \Sigma_k^+$ such that $\alpha + \beta \in \Phi$ or $\langle \alpha, \beta \rangle = -1$.
Notice that by~\eqref{eq:chi-w} one has $\chi_{\omega, X}^n(w_\beta(1)) = w_\beta(X^{n\langle \omega, \beta \rangle})$, therefore
${}^{w_\beta(X^n)} x_{s_\beta(\alpha)}(f) = x_\alpha(\eta X^{-n} f)$
for $\eta = \eta(\beta, s_\beta(\alpha)) = \pm 1$.
Consequently, ${}^{\sigma^n \cdot w_\beta(1) \sigma^{-n}} x_{s_\beta(\alpha)}(f) = x_\alpha(\eta X^{-n} f)$.

Notice also that from~\eqref{eq:w-computation} it follows that
\[ w_\beta(u) \cdot \chi_{\omega, X}(w^{-1}_\beta(u)) = \{ X^{-\langle \omega, \beta \rangle }, u \} \cdot h_\beta(X^{-\langle \omega, \beta \rangle}). \]

Consequently, a direct computation involving~\eqref{eq:sigmadef}, \eqref{eq:sigma-def} and~\cite[Lemme~5.1--5.2]{Ma69}
shows that there exists a central element $c \in \StH(\Phi, A[X, X\inv])$ (in fact, $c$ is a product of Steinberg symbols) such that
\begin{multline*}
[1, 1, z] \rightsquigarrow_{\delta_{-\varpi_k}^n}
[1, 1, \widetilde{\chi}_{-\varpi_k}^n(z)] \rightsquigarrow_{w^{-1}_\beta(1)}
[w_\beta^{-1}(1), 1, \widetilde{\chi}_{-\varpi_k}^n(z)] \rightsquigarrow_{\delta_{\varpi_k}^n} \\
\rightsquigarrow_{\delta_{\varpi_k}^n} [w_\beta^{-1}(1), c \cdot h_\beta^n(X^{-1}), z] \rightsquigarrow_{x_{s_\beta(\alpha)}(f)}
[x_{s_\beta(\alpha)}(f) \cdot w_\beta^{-1}(1), c \cdot h_\beta^n(X^{-1}), z] = \\
= [w_\beta^{-1}(1) \cdot x_{\alpha}(\eta f), c \cdot h_\beta^n(X^{-1}), z] =
[w_\beta^{-1}(1), c \cdot h_\beta^n(X^{-1}), x_{\alpha}(\eta f)^{h_\beta^n(X^{-1})} \cdot z] = \\
= [w_\beta^{-1}(1), c \cdot h_\beta^n(X^{-1}), x_{\alpha}(\eta fX^{-n}) \cdot z] \rightsquigarrow_{\delta_{-\varpi_k}^n} \\
\rightsquigarrow_{\delta_{-\varpi_k}^n} [w_\beta^{-1}(1), 1, \widetilde{\chi}_{-\varpi_k, X}^n (x_{\alpha}(\eta fX^{-n}) \cdot z)] \rightsquigarrow_{ w_\beta(1) } \\
\rightsquigarrow_{ w_\beta(1) } [1, 1, \widetilde{\chi}_{-\varpi_k, X}^n (x_{\alpha}(\eta fX^{-n}) \cdot z)] \rightsquigarrow_{\delta_{\varpi_k}^n}
[1, 1, x_{\alpha}(\eta fX^{-n}) \cdot z] \qedhere
\end{multline*}
\end{proof}

Now we are ready to prove the main result of this section.
\begin{proof}[Proof of~\cref{thm:horrocks-k2}]
Thanks to~\cref{lem:first-reduction} and~\cref{lem:second-reduction} it suffices to show that the action of
$\St(\Phi, A[X, X\inv])$ on $\overline{V} = V/\sim_W$ constructed above satisfies the property
$\mu(h) \cdot [1, 1, z] = [1, 1, hz]$ for $h, z \in \St(\Phi, A[X, X\inv], M[X, X\inv])$.
Clearly, it suffices to verify this property for $h$ belonging to a generating set of $\St(\Phi, A[X, X\inv], M[X, X\inv])$.
According to~\cref{lem:relative-generators} it suffices to verify this assertion for two types of generators:
\begin{itemize}
\item $z_{\beta}(f, g)$, $f \in M[X, X\inv]$, $g \in A[X, X\inv]$, $\beta \in \Sigma_k^-$;
\item $x_\alpha(f)$, $f \in M[X, X\inv]$, $\alpha \in \Phi$.
\end{itemize}
The assertion for the generators of the first type is contained in~\cref{lem:family1}.
The assertion for the generators of second type follows from~\cref{lem:family1} (in the special case $\alpha \in \Sigma_k^+$ and $g = 0$) and~\cref{lem:family2} (for $\alpha \in \Delta_k$).
\end{proof}

\section{Proof of main results} \label{sec:main}
First of all, notice that we already have all technical ingredients to prove
some of our main results.

\begin{proof}[Proof of~\cref{thm:LP-for-K2}]
This is what the proof of~\cite[Theorem~1.1]{LSV2} actually shows if we use the more general~\cref{thm:horrocks-k2}.
\end{proof}

\begin{proof}[Proof of~\cref{cor:motivic-pi1}]
Repeat the proof of~\cite[Corollary~1.2]{LSV2} verbatim.
\end{proof}

The remainder of this section is devoted to the proof of~\cref{cor:dedekind}.
For the rest of this subsection $\Psi$ is an arbitrary irreducible simply-laced root system of rank $\geq 3$ embedded into another irreducible simply-laced root system $\Phi$.
For a commutative ring $R$ we denote by $j_R$ the corresponding homomorphism of Steinberg groups $\St(\Psi, R) \to \St(\Phi, R)$ induced by this embedding.

The following result is generalizes the so-called ``dilation principle'' for subsystem embeddings, cf.~\cite[Corollary~2.4]{Tu83}.
\begin{lemma}\label{lem:dp-2}
Let $h\in\St(\Phi, A[X], X A[X])$ be such that $\lambda_a^*(h) \in \Img(j_{A_a[X]})$.
Then for sufficiently large $n$ one has
\[\ev{A}{A[X]}{a^n\cdot X}^*(h) \in \mathrm{Im}(j_{A[X]}).\]
\end{lemma}
\begin{proof}
We denote by $A\ltimes XA_a[X]$ the semidirect product of $A$ and the ideal $XA_a[X]$, cf. e.\,g.~\cite[Definition~3.2]{S15}.
Denote by $\theta$ the obvious map $A[X]\rightarrow A\ltimes XA_a[X]$ localizing at $a$ all coefficients of terms of degree $\geq 1$.

Recall from~\cite[\S~2]{LS17} that there exists a homomorphism
\[T_\Psi \colon \St(\Psi, A_a[X], XA_a[X]) \to \St(\Psi, A \ltimes XA_a[X])\]
such that $\theta^* = T_\Psi \circ \lambda_a^*$.

Let $(A_i, f_{ij})$ be the directed system of rings given by
\[A_i\coloneqq A[X],\ f_{ij} \coloneqq \ev{A}{A[X]}{a^{j-i} \cdot X},\ 0 \leq i\leq j.\]
It is easy to check that $\varinjlim A_i$ coincides with $A \ltimes XA_a[X]$.
The canonical morphisms $A_j\rightarrow \varinjlim_i A_i \cong A \ltimes XA_a[X]$ can be easily computed as $\ev{A}{A\ltimes XA_a[X]}{a^{-j} \cdot X}$.

By hypothesis $\lambda_a^*(h) = j_{A_a[X]}(h')$ for some $h' \in \St(\Psi, A_a[X], XA_a[X])$.
Consequently, $\theta^*(h) = j_{A \ltimes XA_a[X]}(T_\Psi(h'))$
and the assertion of the lemma follows from the fact that the Steinberg group functor commutes with
colimits over directed systems (cf.~\cite[Lemma~2.2]{Tu83}):
\[\St(\Psi, A\ltimes XA_a[X]) = \St(\Psi, \varinjlim_i A_i) \cong \varinjlim_i \St(\Psi,A_i). \qedhere\]
\end{proof}

Recall that a pair of elements $a, b \in A$ are called \textit{coprime} if together they generate the unit ideal, i.\,e. $Aa + Ab = A$.
The following lemma generalizes~\cite[Lemma~2.5]{Tu83}.
\begin{lemma}\label{lem:L25-2}
Let $a$ and $b$ be a pair of coprime elements of $A$.
Let $g$ be an element of $\St(\Phi, A[X], XA[X])$ such that
$\lambda_a^*(g) \in \Img(j_{A_a[X]})$ and $\lambda_b^*(g) \in \Img(j_{A_b[X]})$.
Then $g$ lies in the image of $j_{A[X]} \colon \St(\Psi, A[X]) \to \St(\Phi, A[X])$.
\end{lemma}
\begin{proof}
We reproduce the argument of~\cite[Lemma~2.5]{Tu83}, cf. also with~\cite[Lemma~16]{S15}.
Set $S \coloneqq A[X, Y]$.
Consider the following element of $\St(\Phi, S[Z])$:
\[h(X, Y, Z) \coloneqq g(YX) \cdot  g^{-1}((Y+Z) X) = \ev{A}{S[Z]}{YX}^* \left(g\right) \cdot \ev{A}{S[Z]}{(Y + Z)X}^*\left(g^{-1}\right).\]
It is easy to see that $h(X, Y, Z)$ belongs to
\[\Ker\left(\eval{Z}{S}{S}{0}^*\colon\St(\Phi, S[Z]) \rightarrow \St(\Phi, S)\right)\]
and hence by~\cite[Lemma~8]{S15} lies in $\St(\Phi, S[Z], Z S[Z])$.
On the other hand, \begin{multline*}
\lambda_{a}^*(h(X, Y, Z)) = \left(\lambda_a\circ \ev{A}{S[Z]}{Y X}\right)^*(g) \cdot \left(\lambda_a\circ \ev{A}{S[Z]}{(Y + Z)X}\right)^*(g^{-1}) = \\
= \ev{A_a}{S_a[Z]}{YX}^*(\lambda_{a}^*(g)) \cdot \ev{A_a}{S_a[Z]}{(Y + Z)X}^*(\lambda_{a}^*(g^{-1})) \end{multline*}
lies in the image of $j_{S_a[Z]}$.
Similarly, $\lambda_{b}^*(h(X, Y, Z))$ lies in the image of $j_{S_b[Z]}$.

We claim that there exists $n$ such that both $h(X, Y, a^n Z)$ and $h(X, Y, b^n Z)$
lie in the image of $j(S[Z])$.

By assumption, there exist $r, s\ \in A$ such that $r a^n + s b^n = 1$, consequently
\begin{multline*}
g(X) = g(X)\cdot g^{-1}(ra^n\cdot X) \cdot g(ra^n\cdot X) \cdot g^{-1}(0) = \\
= h(X, 1, -sb^n) \cdot h(X, ra^n, -ra^n) \in \mathrm{Im}(j_{A[X]}). \qedhere
\end{multline*} \end{proof}

The following result is the subsystem analogue of~\cite[Theorem~2]{LS17} which generalizes~\cite[Theorem~2.1]{Tu83}.
\begin{cor}[Local-global principle for subsystem embeddings] \label{cor:QS-subsystem}
An element $g \in \St(\Phi, A[X], XA[X])$ belongs to $\Img(j_{A[X]})$ if and only if
$\lambda_M^*(g) \in \St(\Phi, A_M[X])$ belongs to $\Img(j_{A_M[X]})$ for all maximal ideals $M \trianglelefteq A$.
\end{cor}
\begin{proof}
It suffices to show ``if'' part of the statement.
One defines the \textit{Quillen set} $Q(g)$ as the set consisting of those elements $a \in A$
such that $\lambda_a^*(g) \in \Img(j_{A_a[X]})$.

Repeating the same argument as in the proof of~\cite[Theorem~2]{S15} one shows using~\cref{lem:L25-2}
that $Q(g)$ is an ideal of $A$, which can not be proper and therefore must coincide with $A$.
\end{proof}

Before we proceed further we would like to briefly recall the main construction from~\cite{LS20} upon which
the proof of Theorem~1 ibid.\ is based.
Recall that in~\cite{LS20} one constructs an action of the group $\St(\Phi, A[X\inv] + M[X])$ on a certain set $\overline{V}$.
The set $\overline{V}$ referenced here is distinct from the set $\overline{V}$ introduced in~\cref{sec:horrocks}.
Since the latter is no longer relevant to our discussion, we will henceforth use the notation $\overline{V}$ exclusively to denote the set defined in~\cite{LS20},
thereby avoiding any potential ambiguity in notation.

The construction of $\overline{V}$ proceeds as follows.
Set $G_{M, \Phi}^{\geq 0} \coloneqq \Img(\St(\Phi, A[X], M[X]) \to \St(\Phi, A[X, X\inv])), G_M^0 \coloneqq \overline{\St}(\Phi, A, M)$.
$G_M^0$ is easily seen to be a subgroup of both $\St(\Phi, A[X\inv])$ and $G_M^{\geq 0}$.
Denote by $\overline{V}$ the factor set of the product $V \coloneqq G_M^{\geq 0} \times \St(\Phi, A[X\inv]) \times (1 + M)^\times$
modulo the equivalence relation given by $(gh_0, h, u) \cong (g, h_0h, u)$ where $h_0 \in G_M^0, (g, h, u) \in V.$
Denote by $[g, h, u] \in \overline{V}$ the equivalence class corresponding to $(g, h, u)\in V$, cf.~\cite[\S~5.4]{LS20}.

Note also that, although the results in~\cite[\S~5.5]{LS20} are conditional,
i.e., they are formulated under the additional assumption that the canonical homomorphism $\St(\Phi, A[X^{-1}] + M[X]) \to \St(\Phi, A[X, X^{-1}])$ is injective,
this does not pose a problem for us, as this condition was already verified in the proof of~\cref{lem:first-reduction}.

\begin{prop} \label{prop:horrocks-main} The group $\St(\Phi, A[X\inv] + M[X])$ acts simply transitively on $\overline{V}$.
This action satisfies the following additional property.
Suppose that for some
\[g \in \Img(j\colon \St(\Psi, A[X\inv] + M[X]) \to \St(\Phi, A[X\inv] + M[X])), \]
$h, h' \in \St(\Phi, A[X\inv])$ and $u, h' \in 1 + M$ one has
\[ g \cdot [1, h, u] = [g', h', u'].\]
Then $g'$ belongs to $\Img(j\colon G_{M, \Psi}^{\geq 0} \to G_{M, \Phi}^{\geq 0})$.
\end{prop}
\begin{proof}
The existence of the action of $\St(\Phi, A[X\inv] + M[X])$ and its faithfullness are contained in
Proposition~5.39 and Remark 5.42 of~\cite{LS20}.
The stated property follows from the construction of the action in~\cite[\S~5.4]{LS20}.
\end{proof}

The following lemma generalizes~\cite[Proposition~4.3(b)]{Tu83}.
\begin{lemma} \label{lem:horrocks-subsystem-local}
Let $A$ be a local domain with maximal ideal $M$ and residue field $\kappa$.
Suppose that the image in $\St(\Phi, A[X, X\inv])$ of the element $x \in \K_2(\Phi, A[X], XA[X])$
can be decomposed as $j_{A[X, X\inv]}(y) \cdot \lambda_{X^{-1}}^*(z)$ for
some $y \in \St(\Psi, A[X, X\inv])$ and $z \in \St(\Phi, A[X\inv])$
Then $x$ belongs to $\Img(j_{A[X]})$.
\end{lemma}
\begin{proof}
We denote by $\rho_{A}$ (resp. $\rho_{A[X]}$, $\rho_{A[X, X\inv]}$) the canonical homomorphism
$A \to \kappa$ (resp. $A[X] \to \kappa[X]$, $A[X, X\inv] \to \kappa[X, X\inv]$).

Recall from~\cite{Hur77} that $\K_2(\Phi, \kappa[X]) = \K_2(\Phi, \kappa)$ hence one has $\rho_{A[X]}(x) = 1$
and by our assumptions $\rho_{A[X, X\inv]}(j_{A[X, X\inv]}(y) \cdot z) = 1$.
Consequently, we can find $z_0 \in \St(\Psi, A[X\inv])$ and $z_1 \in \overline{\St}(\Phi, A[X^{-1}], M[X^{-1}])$
such that $z = j_{A[X\inv]}(z_0) \cdot z_1$.
It is clear that $\rho_{A[X, X\inv]}(y \cdot \lambda_{X^{-1}}^*(z_0)) = 1$ hence we can find
$y' \in \St(\Psi, A[X, X\inv], M[X, X\inv])$ such that $\mu(y') = \cdot \lambda_{X^{-1}}(z_0)$.
Notice also that $t(y') \in \St(\Phi, A[X\inv] + M[X])$, where $t$ denotes the homomorphism discussed in~\cref{lem:first-reduction}.
From~\cref{prop:horrocks-main} we obtain that
\[ [x, 1, 1 ] = t(y') [1, z_1, 1] = [g', h', u']. \]
for some $g' \in \Img(\G_{M, \Psi}^{\geq 0} \to G_{M, \Phi}^{\geq 0})$.
From the definition of $\overline{V}$ we obtain that $x = g' \cdot h_0$ for some $h_0 \in \overline{\St}(\Phi, A, M)$.
But since $x(0) = 1$ we conclude that $x = g' \cdot (g'(0))^{-1}$ from which the assertion of the lemma follows.
\end{proof}

\begin{cor} \label{cor:horrocks--ingredient}
Let $A$ be an arbitrary commutative domain.
Suppose that the image in $\St(\Phi, A[X, X\inv])$ of the element $x \in \K_2(\Phi, A[X], XA[X])$
can be decomposed as $j_{A[X, X\inv]}(y) \cdot \lambda_{X^{-1}}^*(z)$ for
some $y \in \St(\Psi, A[X, X\inv])$ and $z \in \St(\Phi, A[X\inv])$
Then $x$ belongs to $\Img(j_{A[X]})$.
\end{cor}
\begin{proof}
This is a consequence of~\cref{lem:horrocks-subsystem-local} and~\cref{cor:QS-subsystem}.
\end{proof}

We start by recalling the so-called Zariski excision property of Steinberg groups.
\begin{lemma} \label{lem:zariski-glueing}
Let $\Phi$ be any simply-laced root system of rank $\geq 3$.
Let $A$ be a commutative domain and $a, b \in A$ be a pair of coprime elements.
\begin{enumerate}
\item Let $\delta$ be an element of $\St(\Phi, A_{ab})$.
Then $\delta$ can be presented as $\lambda_b^*(x) \cdot \lambda_a^*(y)$ for some
$x  \in \St(\Phi, A_a)$ and $y \in \St(\Phi, A_b)$.
\item  Let $x \in \St(\Phi, A_a)$ and $y \in \St(\Phi, A_b)$ be such that the equality $\lambda_b^*(x) = \lambda_a^*(y)$ holds in $\St(\Phi, A_{ab})$.
Then there exists $z \in \St(\Phi, A)$ such that $x = \lambda_a^*(z)$, $y = \lambda_b^*(z)$.
\end{enumerate}
\end{lemma}
\begin{proof}
See~\cite[Corollary~4.5]{LSV2}, cf.\ also the proof of~\cite[Lemma~2.6]{LSV2}.
\end{proof}

\begin{lemma} \label{lem:horrocks-b}
Let $A$ be a commutative domain and $f \in A[X]$ be a unitary polynomial.
Let $h$ be an element of $\K_2(\Phi, A[X], XA[X])$ such that $\lambda_f^*(h)$ belongs to the image of the stabilization map
$\St(\Psi, A[X]_f) \to \St(\Phi, A[X]_f)$.
Then $h$ lies in the image of $j_{A[X]}\colon\St(\Psi, A[X]) \to \St(\Phi, A[X])$.
\end{lemma}
\begin{proof}
Suppose that $f(X) = X^n + a_1 X^{n-1} \ldots + a_n$.
Set $g(X\inv) = 1 + a_1 X\inv + \ldots + a_{n} X^{-n}$.
It is clear that $A[X, X\inv]_f = A[X\inv]_{X\inv g}$ and, moreover, that $X\inv$ and $g$ are not zero divisors and together generate the unit ideal of $A$.

Consider the following diagram:
%! suppress = EscapeAmpersand
\[ \xymatrix{\St(\Phi, A[X]) \ar[r]^{\lambda_X} \ar[d]_{\lambda_f^*} & \St(\Phi, A[X, X\inv]) \ar[d]_{\lambda_f^*}  & \St(\Phi, A[X\inv]) \ar[l]_{\lambda_{X\inv}^*} \ar[d]_{\lambda_g^*}  \\
\St(\Phi, A[X]_f) \ar[r] & \St(\Phi, A[X, X\inv]_f) & \St(\Phi, A[X\inv]_g) \ar[l] \\
\St(\Psi, A[X]_f) \ar[r] \ar[u]_j & \St(\Psi, A[X, X\inv]_f) \ar[u]_j & \St(\Psi, A[X\inv]_g) \ar[l] \ar[u]_j \\
& \St(\Psi, A[X, X\inv]) \ar[u]_{{\lambda_f^\Psi}^*}   & \St(\Psi, A[X\inv]). \ar[l] \ar[u]_{{\lambda_g^\Psi}^*}}\]

By assumption there exists $\widetilde{h} \in \St(\Psi, A[X]_f)$ such that $j(\widetilde{h}) = \lambda_f(h)$.
By the first part of~\cref{lem:zariski-glueing} one can write
$\lambda_X(\widetilde{h}) = {\lambda_f^\Psi}^*(z) \cdot \lambda_{X\inv}^*(y)$
for some $y \in \St(\Psi, A[X\inv]_g)$, $z \in \St(\Psi, A[X, X\inv]).$
Consequently, one has $\lambda_f^*(j(z)\inv \cdot \lambda_X^*(\alpha)) = \lambda_{X\inv}^*(j(y))$.
By the second part of~\cref{lem:zariski-glueing} there exists $y' \in \St(\Phi, A[X\inv])$ such that $\lambda_X^*(\alpha) = j(z) \cdot \lambda_{X\inv}^*(y').$
The assertion now follows from~\cref{cor:horrocks--ingredient}.
\end{proof}

\begin{thm}\label{thm:early-stability}
Let $\Phi$ be a root system of type $\rA_{\geq 4}$, $\rD_{\geq 5}$ or $\rE_{6,7,8}$ and let $A$ be an arbitrary noetherian commutative
domain of Krull dimension $\leq 1$.
Then for any $n \geq 0$ the obvious inclusion $\rA_4 \subseteq \Phi$ induces a surjection
\[\K_2(\rA_4, A[X_1,\ldots, X_n]) \to \K_2(\Phi, A[X_1, \ldots X_n]).\]
\end{thm}
\begin{proof}
The proof is modeled after the proof of~\cite[Theorem~5.3]{Tu83}.
We proceed by induction on $n$.
Our assumption on the dimension of $A$ guarantees that it satisfies the condition $\mathrm{SR}_3$ in the sense of~\cite{St78}.
Thus, from Corollary~3.2 and Theorem~4.1 of~\cite{St78} we conclude that the composite arrow in the following diagram is a surjection:
\[\K_2(\rA_2, A) \to \K_2(\rA_4, A) \to \K_2(\Phi, A).\]
Consequently, we obtain that the right arrow is a surjection, which yields the induction base.

Now let us verify the induction step.
Set $C = A[X_2, \ldots , X_n]$ and $B = C[X_1]$.
We need to show that $\K_2(\rA_4, B) \to \K_2(\Phi, B)$ is surjective.
Every element $\alpha \in \K_2(\Phi, B)$ can be decomposed as a product $\alpha = \alpha_0 \cdot \alpha_1$,
where $\alpha_0 \in \K_2(\Phi, C)$ and $\alpha_1 \in \K_2(\Phi, B, X_1 B)$.
By inductive assumption $\K_2(\rA_4, C)$ surjects onto $\K_2(\Phi, C)$, so it remains to show that $\alpha_1$ lies in the image of $\K_2(\rA_4, B)$.

Denote by $S$ the multiplicative system $S \subseteq B$ consisting of polynomials $f$ such that for sufficiently large $m$
the polynomial $f$ becomes unitary in $Y_1$ after substitutions $X_1 \coloneqq Y_1,$ $X_2 \coloneqq Y_2 + Y_1^m, \ldots X_n \coloneqq Y_n + Y_1^{m^n}$.
Recall from~\cite[\S~6]{Su77} that $\dim(S^{-1}B) \leq \dim(A) = 1$.
By induction base the map $\K_2(\rA_4, S^{-1}B) \to \K_2(\Phi, S^{-1}B)$ is surjective.
Since functor $\K_2$ commutes with filtered colimits (cf. \cite[Lemma~3.3]{LSV2}) there exists $f \in S$ such that $\lambda^*_f(\alpha_1)$ lies in the image of $\K_2(\rA_4, B_f) \to \K_2(\Phi, B_f)$.
By the construction of $S$ we may assume that $f$ is unital in $X_1$.
The required assertion now follows from~\cref{lem:horrocks-b}.
\end{proof}

\begin{proof}[Proof of~\cref{cor:dedekind}]
Consider the following diagram:
\[ \K_2(\rA_3, A) \to \K_2(\rA_4, A) \to \K_2(\rA_4, A[X_1, \ldots, X_n]) \rightarrow \K_2(A[X_1, \ldots, X_n]) \to \K_2(A).\]
The two right arrows on the above diagram are isomorphisms by the $\mathbb{A}^1$-invariance of the stable $\K_2$
(see e.\,g. \cite[Theorem~V.6.3]{Kbook}) combined with the main result of~\cite{Tu83}.
On the other hand, by~\cite[Corollary~3.2]{ST76} the left arrow and the composite arrow are isomorphisms.
Our assertion now follows from~\cref{thm:early-stability}.
\end{proof}

\printbibliography
\end{document}